 \documentclass[final,1p,times]{elsarticle}

 \usepackage{graphicx}

\usepackage{lineno,hyperref}
\modulolinenumbers[5]
\usepackage{amssymb}
\usepackage{color}
\newtheorem{remark}{Remark}
\newtheorem{theorem}{Theorem}
\newtheorem{lem}{Lemma}

\newtheorem{cor}{Corollary}

\newproof{proof}{Proof}
\newproof{pot}{Proof of Theorem \ref{ENMTW1}}

\let \al=\alpha
\let \be=\beta

\let \vare=\epsilon

\let \q=\quad

\newcommand{\R}{{\Bbb R}}

\newcommand{\N}{{\Bbb N}}
\newcommand{\C}{{\Bbb C}}
\begin{document}
\begin{frontmatter}

\title{Traveling waves in the 
nonlocal KPP-Fisher equation: 
different roles of the right and the left interactions}

\author[a]{Karel Hasik}
\author[a]{Jana Kopfov\'a}
\author[a]{Petra N\'ab\v{e}lkov\'a}
\author[c]{Sergei Trofimchuk}
\address[a]{Mathematical Institute, Silesian University, 746 01 Opava, Czech Republic}
\address[c]{Instituto de Matem\'atica y F\'isica, Universidad de Talca, Casilla 747,
Talca, Chile }

\bigskip

\begin{abstract}
\noindent  We consider  the nonlocal KPP-Fisher equation $u_t(t,x) = u_{xx}(t,x)  + u(t,x)(1-(K *u)(t,x))$ which describes the evolution of  population density $u(t,x)$ with respect to time $t$ and location $x$. The non-locality 
is expressed in terms of  the convolution of $u(t, \cdot)$ with  kernel $K(\cdot) \geq 0,$  $\int_{\R} K(s)ds =1$. The 
restrictions $K(s), s \geq 0,$ and $K(s), s \leq 0,$ are responsible for 
interactions of an individual with his left and right neighbors, respectively.  We show that these two parts of $K$ play quite different roles 
as  for  the existence and uniqueness of traveling fronts to the KPP-Fisher equation. In particular,  
if  the left interaction is dominant,   the uniqueness of fronts  can be proved, while 
the dominance  of the right interaction can induce the co-existence of 
monotone and oscillating fronts.  We also present a short proof of 
the existence of  traveling waves  without assuming  various technical  restrictions usually imposed on $K$.\end{abstract}
\end{frontmatter}
\newpage

\section{Introduction and main results}
\label{intro} 
This paper continues the studies of  traveling waves for the following  nonlocal version \cite{AC,BNPR,ChSh,FZ,GL,NPT,NPTT,ST,VP,VPb} of the  KPP-Fisher equation:
\begin{equation}\label{17nl}
u_t(t,x) = u_{xx}(t,x)  + u(t,x)(1-(K *u)(t,x)), \quad u \geq 0,\ (t,x) \in \R^2.
\end{equation}
The requirement $u(t,x) \geq 0$  is due to  the usual  interpretation of $u(t,x)$ as the population density at time $t$ and location $x$. The  convolution $(K *u)(t,x):= \int_{-\infty}^{+\infty} K(y) u(t,x-y)dy$ describes
the non-local interaction among  individuals; it is assumed that  the non-negative kernel $K \in L^1(\R, \R_+)$ is normalised by  $|K|_1= \int_\R K(s) ds=1$. It is clear that the restriction of $K(s)|_{\{s \geq 0\}}$ characterizes the instantaneous interaction of an individual  with his left side neighbors, its  intensity $\alpha_- \in [0, +\infty]$
  can be expressed  as  
$$
\alpha_-= \frac{1}{c}\int_0^{+\infty} sK(s)ds, 
$$
where $c$ is some positive parameter (wave velocity) to be specified later. Similarly, 
$$
\alpha_+ = \frac{1}{c}\int^0_{-\infty} |s|K(s)ds
$$
 can be used to quantify the intensity of  the interaction of an individual  with his right side neighbors.

We recall that the classical solution $u(t,x) = \phi(x +ct)$  is a
wavefront (or a traveling front) for  (\ref{17nl}) propagating with the velocity $c\geq 0$, if the profile
$\phi$ is $C^2$-smooth, non-negative and satisfies $\phi(-\infty) = 0$ and
$\phi(+\infty) = 1$. By replacing the condition $\phi(+\infty) = 1$ with the less restrictive  condition
$0< \liminf_{s \to +\infty}\phi(s) \leq \limsup_{s \to +\infty}\phi(s) < \infty$, we get the definition of a semi-wavefront. 
Clearly, each wave profile $\phi$ to (\ref{17nl})  satisfies the functional differential equation  
\begin{eqnarray} \label{twe2an} 
\phi''(t) - c\phi'(t) + \phi(t)(1-
(\phi* K)(t)) =0,  \quad t \in \R. \end{eqnarray}

The main concern of  this paper is the existence and uniqueness of wavefronts and semi-wavefronts to equation 
(\ref{17nl}) in the situation when $\alpha_+ >0$.  Since we have much more information about the existence-uniqueness  problem 
when $\alpha_+ =0$ (i.e. in the so-called delayed case), it is enlightening   to recall here the key results about traveling waves for the delayed KPP-Fisher equation:
\subsection{Case $\alpha_+ =0$: expected uniqueness of traveling fronts in the Hutchinson diffusive equation.} \label{Sub11}
If we formally choose  $K(t)= \delta(t- c\tau)$ with some $\tau >0$,  then (\ref{twe2an}) takes 
the form 
\begin{eqnarray} \label{twe2ade} 
\phi''(t) - c\phi'(t) + \phi(t)(1-\phi(t-c\tau)) =0,  \quad t \in \R, \end{eqnarray}
which is precisely the wave profile's equation for   the
diffusive Hutchinson's model 
\begin{equation}\label{17} \hspace{5mm}
u_t(t,x) = u_{xx}(t,x)  + u(t,x)(1-u(t-\tau,x)), \ u \geq 0,\ x \in
\R. \end{equation} 
Model (\ref{17}) is an important
example of delayed reaction-diffusion equations. In particular, during the past decade, 
the traveling fronts for this model  have been analysed by many authors, see 
\cite{ZAMP,BGHR,ADN,FW,fhw,FTnl,GT,GTLMS,HTa,HTb,KO,wz}. 
As a result  of these studies,  nowadays there is a rather satisfactory understanding 
of  the wavefronts' existence and uniqueness problems for model (\ref{17}) and, more generally, for equation (\ref{17nl}) with $\alpha_+ =0$, cf. \cite{FTnl}. It should 
be noted here that we are still far from having the complete solution to these problems: 
nevertheless,  several key open questions and plausible answers to them are stated 
in  \cite{ADN,HTa,HTb}.  In particular, 
the decomposition of the domain of parameters $(\tau,c) \in \R^2_+$ on the disjoint subsets 
associated with the classes of monotone wavefronts, non-monotone wavefronts, proper semi-wavefronts and 
no of semi-wavefronts to equation (\ref{17})  was obtained, modulo the generalized Wright conjecture \cite{BCKN,ADN,HTa,HTb}. By \cite{HTa},   for each $c \geq 2$ equation (\ref{17})  possesses at least one semi-wavefront. 
The uniqueness of the monotone wavefronts to  (\ref{17}) was proved in \cite{FW, GT, HTa}.  Moreover, 
a combination of \cite[Theorem 1.1 and Corollary 6.6]{fhw} with \cite[Theorem 5.1]{FTnl} assures the uniqueness of all fast (this means  $c\gg 1$) wavefronts for $\tau \leq 3/2$.  Actually, \cite{fhw} suggests that the uniqueness of all fast  semi-wavefronts can be deduced from the uniqueness of the heteroclinic connection in the Hutchinson's equation.  Since the proper semi-wavefronts are slowly oscillating \cite{ADN,HTa}, an expected positive answer 
to Jones' conjecture \cite{BCKN} (the uniqueness of slowly oscillating periodic solution in the Wright equation) gives an additional argument in favor  of the uniqueness of  semi-wavefronts for equation  (\ref{17}). 
Hence,  we believe that for each fixed pair $(\tau,c)$, $\tau \geq 0$, $c \geq 2$, the semi-wavefront solution to equation 
(\ref{17}) is unique (up to a translation).

\subsection{Case $\alpha_+ >0$: main existence and convergence results.}
It is somewhat surprising that the first  existence result for the equation (\ref{17nl})  was proved
under condition $\alpha_+ >0$. More precisely, it was established  by Berestycki {\it et al.}  \cite{BNPR} that the assumptions 
\begin{equation} \label{bc}
c \geq 2 \quad \mbox{and} \quad K \in C^1(\R, \R_+), \ K(0) >0, \  |K|_1=  1, \ \int_{-\infty}^{+\infty} K(s) e^{\lambda (c)s} ds <\infty, 
\end{equation}
where
\begin{equation}\label{E}
\lambda(c): = \frac{1}{2}(c- \sqrt{c^2-4})  \leq \mu(c): =  \frac{1}{2}(c+\sqrt{c^2-4}) 
\end{equation}
denote the positive roots of the quadratic equation $z^2 -cz + 1 = 0$, 
guarantee the existence of at least one semi-wavefront of   (\ref{17nl}).  Observe that the last inequality in (\ref{bc}) does not appear explicitly in \cite[Theorem 1.4]{BNPR}, however it was used to construct  a super-solution, cf. \cite[ p. 2836]{BNPR}.   Note also that  the  condition  $K(0) >0$ of (\ref{bc}) is essential for the proofs in \cite{BNPR} and  therefore  the existence result from \cite{BNPR} cannot be applied  when $\alpha_+ =0$ or $\alpha_- =0$.  Thus the known proofs \cite{ADN,HTa} of the existence of semi-wavefronts for  (\ref{17}) are based on rather different approaches.  

We show in the present paper that the method of \cite{HTa} can be also applied to  (\ref{17nl}) which allows  
 to weaken restrictions (\ref{bc}):
\begin{theorem} \label{Te1} Assume that  $K\in L_1(\R,\R_+), \ |K|_1= 1$. Then equation (\ref{17nl})  has at least one semi-wavefront  $u(t,x) = \phi(x+ct)$ if and only if $c \geq 2$.   \end{theorem}
It is not difficult to deduce from this result the existence of at least one semi-wavefront for each given velocity $c \geq 2$ in the case of a more general equation   
\begin{eqnarray} \label{twe2anG} 
\phi''(t) - c\phi'(t) + \phi(t)\left(1- \int_{-\infty}^{+\infty}\phi(t-s)dm(s)\right) =0,  \quad t \in \R. \end{eqnarray}
Here the increasing function $m:\R \to \R$ satisfies $m(-\infty)=0,$ $m(+\infty)=1$.  In other words,  
the convolution $K*u$ of a continuous function $u$ with Lebesgue's integrable kernel $K$  (as in equation (\ref{twe2an})) is replaced here by a convolution $u*\mu$ of $u$ with the normalised Borel measure $\mu$ (where $\mu(A) =\int_Adm(s)$).  Clearly, this family of equations includes (\ref{twe2ade}) as a particular case. 

The symmetry (evenness) properties  of the kernel $K$ do not matter for such a general existence result as Theorem \ref{Te1}. However, the shape of $K$ plays a decisive role in the determination of monotone wavefronts to (\ref{17nl}). This question was exhaustively answered by Fang and Zhao \cite{FZ} in terms of roots of the equation  
\begin{equation}\label{fzc} 
\lambda^2-c\lambda - \int_{-\infty}^{+\infty} K(s)e^{-\lambda s}ds =0. 
\end{equation}
By \cite{FZ}, model (\ref{17nl}) has at least one monotone wavefront if and only if equation (\ref{fzc}) has a 
negative root. Moreover, the uniqueness of this wavefront was established  in the class of {\it  all monotone wavefronts}.  One of the main results of this paper shows that the above words in italic cannot be omitted if $\alpha_+>0$. This  makes a  striking difference with equation (\ref{17}) (case $\alpha_+ =0$) where the uniqueness of a monotone wavefront in the class of {\it  all semi-wavefronts} was established.  One of the amazing  consequences of the Fang and Zhao criterion  \cite{FZ} for the case $\alpha_- =0$ is the presence   of a unique monotone wavefront to equation (\ref{17nl})  for each given velocity $c \geq 2$.

Now, contrary to the cases of  proper semi-wavefronts and monotone wavefronts,  the existence and uniqueness of non-monotone wavefronts to equation (\ref{17nl}) with $\alpha_+ >0$ is  largely  an open problem.  The known results in this direction were obtained in \cite{AC,BNPR}.  In particular, Berestycki {\it et al.} \cite{BNPR} proved  that the positivity of the Fourier transform of $K$ (that, in turn, implies that the kernel $K$ is an even function satisfying $K(0) \geq K(s)$ for all $s \in \R$) implies the convergence of all semi-wavefront profiles: $\phi(+\infty)=1$.  The second result due to Alfaro and Coville \cite{AC} was obtained by means of $L^2$-estimates. This technique does not take into account the symmetry properties of $K$: Alfaro and Coville's   theorem  says that  the inequality 
\begin{equation}\label{AlC}
c > M^*\sqrt{\int_{-\infty}^{+\infty}  s^2K(s)ds}, 
\end{equation}
with $M^*$ being any  a priori estimate for the norm $|\phi|_\infty = \sup_{s\in \R}\phi(s)$ of semi-wavefront $\phi$,  $|\phi|_\infty\leq M^*$, 
guarantees that $\phi(+\infty)=1$. It should be noted here that the derivation of  the explicit formulas for $M^*$ is an important step of  proofs of the existence theorems.  The first such formula was proposed in \cite{BNPR} and  Lemma \ref{ogran} below develops further this investigation.   The presence of $M^*$ in  (\ref{AlC}) marks another  difference with the convergence criteria for the case $\alpha_+ =0$. Our analysis in this paper suggests that if  $\alpha_+ >0$,   the dependence of the convergence conditions on  the a priori estimates for $\phi$  cannot be avoided.
\begin{theorem} \label{ENMTW1} Let $M^*$ be a priori estimate for the norm $|\phi|_\infty = \sup_{s\in \R}\phi(s)$ of semi-wavefront $u(t,x)= \phi(x+ct)$,  $|\phi|_\infty\leq M^*$.  Then $\phi(+\infty) =1$ if at least one of the following three conditions is satisfied: 
\begin{enumerate}
\item[1)] $c > M^*\int_{-\infty}^{+\infty} |s|K(s)ds$  \quad \mbox{(i.e. $M^*(\alpha_++\alpha_-) <1$)};
\item[2)] $K(s) = 0$ for $s \leq 0,$ and $c > 2\int_{-\infty}^{+\infty} |s|K(s)ds$
 \quad \mbox{(i.e. $\alpha_+ =0, \ \alpha_- \in (0,1/2)$)};
 \item[3)] $K(s) \not\equiv 0$ for $s \leq 0$, $c > 2\int_{-\infty}^{+\infty} |s|K(s)ds$  \ \  \mbox{(i.e. $\alpha_+ >0, \ \alpha_+ + \alpha_- \in (0,1/2)$)}, and 
 \begin{equation}\label{estm}
M^* < \frac{1+\alpha_+-\alpha_- + \sqrt{(1+\alpha_+-\alpha_-)^2-4\alpha_+}}{2\alpha_+} .\end{equation}
\end{enumerate}
\end{theorem}
We  note  that 
the right hand side of (\ref{estm})  is well defined when $\alpha_+ + \alpha_- \in (0,1/2)$. Condition (\ref{estm}) can be further improved within our approach, however, we 
do not pursue this goal in the paper. It is worth noting that $\alpha_+$ and $\alpha_-$ are entering  (\ref{estm})  in asymmetric way and  this inequality is satisfied automatically  when $\alpha_+ \to 0^+$ (thus  condition 2 of Theorem \ref{ENMTW1}  can be considered as a limit case, at  $\alpha_+ = 0^+$, of condition 3).   Obviously,   the inequality $c > M^*\int_{-\infty}^{+\infty}  |s|K(s)ds$  is less restrictive than the Alfaro and Coville  condition (\ref{AlC}) in view of H\"older's inequality. 

Since the proof of Theorem \ref{ENMTW1} is one of the principal motivations for our studies exposed in the next subsection, we outline it below.  
\begin{pot} Take $c \geq 2$ 
and consider semi-wavefront solution $u(t,x)=\phi(x+ct)$.  Then 
it can be proved that
$
p := \liminf_{t \to +\infty} \phi(t) \leq 1, \quad P := \limsup_{t \to +\infty} \phi(t) \geq 1
$
are positive numbers satisfying certain systems of inequalities, the simplest of which has the following form (see Lemma \ref{L9} in Section \ref{ETW}):
$$
p+\alpha_+ P(1-p) +\alpha_- P(P-1) \geq 1,
$$
$$
 P-\alpha_+ P(P-1) -\alpha_- P(1-p) \leq 1. 
$$
\begin{figure}[h]\label{F1}
\begin{center}
\includegraphics[scale=0.65]{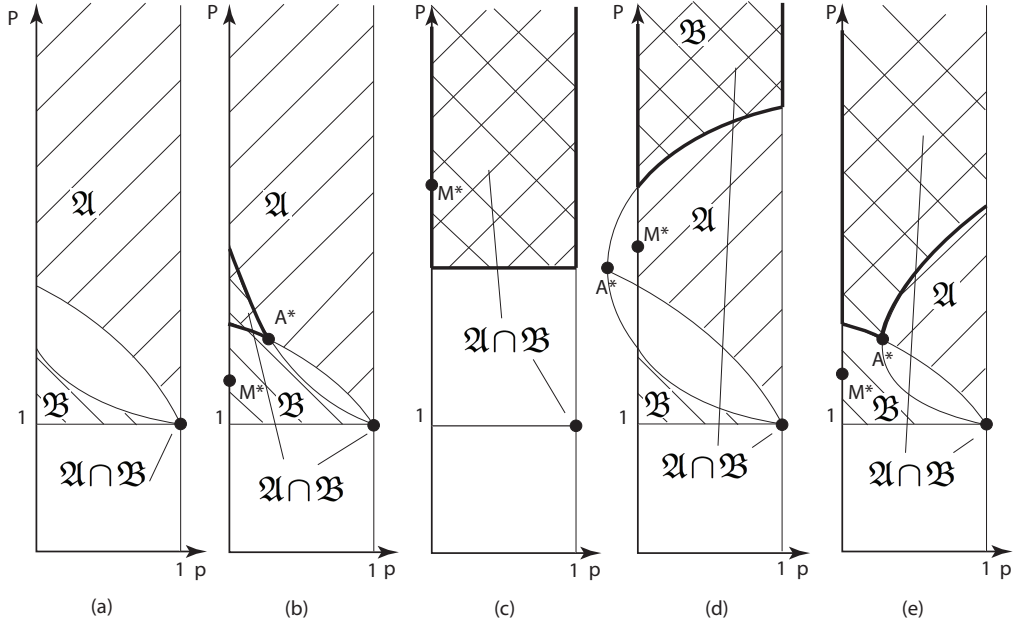}
\end{center}
\caption{Domains $\frak{A}, \ \frak{B}$ and  $\frak{A}\cap \frak{B}$ when (a) $\alpha_+ =0, \ \alpha_- \in (0,1/2)$; (b) $\alpha_+ =0, \ \alpha_- \geq 1/2$;  (c)  $\alpha_- = 0, \ \alpha_+ >0$; (d)  $\alpha_\pm > 0, \ \alpha_+ +\alpha_- <1/2$; (e)  $\alpha_\pm > 0, \ \alpha_+ +\alpha_- \geq 1/2$.}
\end{figure}
Figure 1 represents the position of the domains defined by the first ($\frak{A}$) and the second ($\frak{B}$)
inequality  in the cases  (a) $\alpha_+ =0, \ \alpha_- \in (0,1/2)$; (b) $\alpha_+ =0, \ \alpha_- \geq 1/2$;  (c)  $\alpha_+ > 0, \ \alpha_- =0$; (d)  $\alpha_\pm > 0, \ \alpha_+ +\alpha_- <1/2$; (e)  $\alpha_\pm > 0, \ \alpha_+ +\alpha_- \geq 1/2$,  respectively.  The points $(1,1)$ and  
$$
A^*= \left(2-\frac{1}{\alpha_-+\alpha_+}, \frac{1}{\alpha_-+\alpha_+}\right)
$$
belong to the intersection of the boundaries of $\frak{A}, \ \frak{B}$:  $(1,1)$, $A^* \in \partial  \frak{A} \cap \partial  \frak{B}$. 

In the case (a), it is clear that the unique point satisfying both inequalities is $p=P=1$, which  implies the convergence of each profile at $+\infty$. In the cases (b)-(e), however, the final result depends strongly on the position of $M^*$. If $M^*$ is situated as in the picture (d) (that  is analytically  expressed by (\ref{estm})) or as in pictures (b) and (e)  (that is,   $M^* < 1/(\alpha_++\alpha_-) $), then the upper part of the intersection $\frak{A}\cap\frak{B}$ can be ignored so that $p=P=1$ and we obtain the convergence of all semi-wavefronts at $+\infty$. \qed
\end{pot}
However, if the position of $M^*$ is as in the picture (c), there is a  possibility of 
the co-existence of a monotone wavefront (recall here that $\alpha_-=0$ assures its existence in virtue of Fang and Zhao criterion)  and a proper oscillating semi-wavefront.  The main result of this paper  consists precisely of the analytical proof of such a  dynamical behaviour for certain systems with $\alpha_- =0$.  
\begin{remark}\label{ira} Clearly, the statement of Theorem \ref{ENMTW1} remains true if 
we replace $M^*$ with the smaller value $P = \limsup_{t \to +\infty} \phi(t)$.  In the case (b),  the condition 
$P< 1/(\alpha_++\alpha_-)$ can be replaced by the dual inequality $p > 2- 1/(\alpha_++\alpha_-)$, where 
$p = \liminf_{t \to +\infty} \phi(t)$.

\end{remark}

\subsection{Case $\alpha_- =0$: the co-existence of monotone traveling fronts and proper semi-wavefronts in the  KPP-Fisher equation with advanced argument.}

\noindent The recent work by Nadin {\it et al.} \cite{NPTT}  has provided another  argument supporting the conjecture about the co-existence of different dynamical patterns in  equation (\ref{17nl}). The authors of \cite{NPTT}  have proposed the following substitute, with $K(s)= \delta(s+h)$, (called {\it "a toy model"}) of (\ref{twe2an}):  
\begin{eqnarray}\label{Gt}
\phi''(t) -c\phi'(t) = - \left\{\begin{array}{cc} A\phi(t),& \phi(t) \in [0,\theta), \\    1 - \phi(t+h),
& \phi(t) \geq  \theta,\end{array}\right. 
\end{eqnarray}
(actually, this equation is obtained from the original toy model from \cite{NPTT} by reversing time).  The 
positive parameters $A, h$ and $\theta \in (0,1)$ satisfy the inequality $A \geq (1-\theta)/\theta$, which is the reminiscence of the sub-tangency condition at $0$  of the classical KPP-Fisher nonlinearity.  The piece-wise linear model (\ref{Gt}) inherits the local properties at the steady states from  (\ref{17nl})  and therefore it can be used 
to understand the geometry of the semi-wavefronts to (\ref{17nl}).  It is  a remarkable fact  that the computations of \cite{NPTT} predicted the co-existence of  asymptotically periodic  semi-wavefronts and monotone as well as oscillating wavefronts in equation (\ref{17nl}). Nevertheless, the toy model (\ref{Gt}) has 
one  important deficiency: the right hand side of (\ref{Gt}) is a discontinuous functional.  At a first glance, precisely this drawback could be considered as a main reason for the existence of multiple semi-wavefronts.  Indeed, let us consider the following "delayed" toy model:
\begin{eqnarray}\label{Gt2}
\phi''(t) -c\phi'(t) = - \left\{\begin{array}{cc} \phi(t),& \phi(t) \in [0,0.5), \\    1 - \phi(t-c\tau),
& \phi(t) \geq  0.5,\end{array}\right. 
\end{eqnarray}
where $c =2.5$,  $c\tau = 2\ln1.5 =0.8109\dots$ (so that $\tau = 0.8\ln1.5 = 0.3243\dots < 1/e =0.3678\dots$). It is easy to check that the eigenvalues of (\ref{Gt2}) at $0$ are $0.5$ and $2$, while the set of all eigenvalues at $1$ contains  two negative numbers $-0.5$ and $-4.035\dots$ This information allows us to construct two different monotone wavefronts $\phi_j \in W^{2,\infty}(\R)$ to (\ref{Gt2}): 
\begin{eqnarray*}
\phi_1(t) =  \left\{\begin{array}{cc} 0.5 e^{0.5 t},&  \\    1 - 0.5 e^{-0.5t},
& \end{array}\right.   \phi_2(t) =  \left\{\begin{array}{cc} 0.5 e^{2 t},& t \leq 0, \\    1 - 0.28.. e^{-0.5t}- 0.21..e^{-4.03.. t},
& t >0.\end{array}\right. 
\end{eqnarray*}
Even more surprisingly, an oscillating wavefront to (\ref{Gt2}) can also be constructed. Indeed, since 
$x_0\pm iy_0, \ x_0=-6.2402\dots, \ y_0 = 10.054\dots$ is a pair of conjugated eigenvalues to the  equation (\ref{Gt2}) at the steady state $1$, it is easy to find the following oscillating profile $\phi_3 \in W^{2,\infty}(\R)$:
\begin{eqnarray*}
  \phi_3(t) =  \left\{\begin{array}{cc} 0.5 e^{2 t},& t \leq 0, \\    1 + \hat{a} e^{x_0t}\cos(y_0t + z_0),
& t >0,\end{array}\right.
\end{eqnarray*}
where 
$$
\hat{a}  = \frac{1}{4}+\frac{(1-0.5x_0)^2}{y_0^2} = 0.546\dots, \quad \cos(z_0) = -\frac{0.5}{\hat{a} }, \ z_0 = 2.727\dots  
$$
See Figure 2 where all three solutions are shown. 
\begin{figure}[h]\label{F2}
\begin{center}
\includegraphics[scale=0.37]{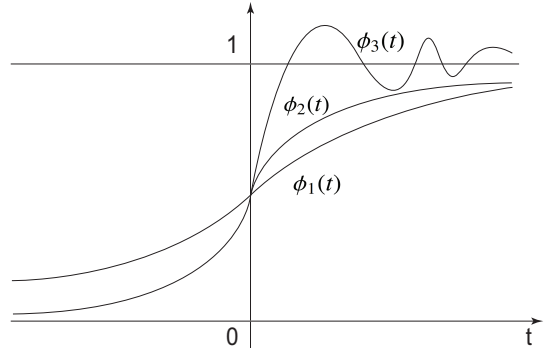}
\end{center}
\caption{Co-existence of monotone and oscillating wavefronts in a delayed toy model.}
\end{figure}
However, in view of the results 
mentioned in Subsection \ref{Sub11}, this equation should possess a unique wavefront (up to a translation).  Moreover,  the wavefront $\phi_2$ decreases rapidly at $-\infty$ (i.e. $\phi_2$ is a pushed front) that is  formally not 
compatible with the above mentioned sub-tangency condition $A \geq (1-\theta)/\theta$.  It is clear that the discontinuity of equation (\ref{Gt2}) is the main reason of  all these "contradictions".  

\vspace{2mm}

Hence the conclusions suggested by the analysis of the "toy" models must be corroborated  by rigorous  analytical proofs. In the present work, using Hale--Lin method \cite{HL} adapted for the singular functional differential equations in \cite{fhw,FTnl,GTLMS};  Hale--Huang analysis of the perturbed periodic solutions developed in \cite{DH,Hale, hale,HW};  Krisztin--Walther--Wu theory of an  invariant stratification of an attracting set for delayed monotone positive feedback \cite{KWW};  Magalh\~aes--Faria  normal forms for retarded functional-differential equations \cite{FM} and  Mallet-Paret--Sell theory of monotone cyclic feedback systems with delay \cite{mps,mps2}, we provide such a  result: 
\begin{theorem} \label{Te2} For each $\tau > 3\pi/2$ sufficiently close to $3\pi/2$ there exists $c_*(\tau)>2$ and an open subset $\Omega$\, of\,  $\R^3$ such that 
the KPP-Fisher equation with advanced argument 
$$
u_t(t,x) = u_{xx}(t,x)  + u(t,x)(1-u(t+\tau,x)), \ u \geq 0,\ x \in
\R,  $$  
has a three-dimensional family $u(t,x) = \phi( x+ct, \zeta,c), \ \zeta \in \Omega,$ of wavefronts for each $c > c_*(\tau)$.   For every fixed $c$, this family 
contains a unique (up to a translation)  monotone wavefront 
and maps continuously  and injectively $\Omega$ into the space $C_b(\R,\R)$ of bounded continuous functions on $\R$.    Moreover,  for each  $c > c_*(\tau)$, the above equation possesses  proper semi-wavefronts $u(t,x) = \psi( x+ct, c)$. The profiles $\psi(\cdot, c)$ are asymptotically  periodic at $+\infty$,  with $\omega(c)$-periodic limit functions $\psi_\infty(\cdot, c)$ having periods $\omega(c)$ close to $2\pi c$ and of the sinusoidal form (i.e.  each $\psi_\infty(\cdot, c)$ oscillates around 1 and has exactly two critical points on the period interval $[0, \omega(c))$). 
\end{theorem}
Theorem \ref{Te2} shows that the non-local KPP-Fisher equations with $\alpha_+ \gg \alpha_-$ may exhibit multiple patterns of wave propagation:  
\begin{cor} \label{coroc}There exists $c >2$ and an increasing function $m:\R \to \R$ satisfying $m(-\infty)=0,$ $m(+\infty)=1$ such that equation (\ref{twe2anG}) has, at the same time, a unique monotone wavefront, multiple oscillating wavefronts as well as asymptotically periodic proper semi-wavefronts propagating with the velocity $c$. 
\end{cor}

The structure  of
this paper is as follows.  In  Section \ref{ETW} we 
establish a series of auxiliary results and a priori estimates necessary to prove Theorems  \ref{Te1} and \ref{ENMTW1}.  Section \ref{ESW} contains the proof of Theorem \ref{Te1}. The first part of Theorem
\ref{Te2} (stated as Theorem \ref{Te2r}) is proved in Sections \ref{FP}, \ref{S5}.  The second part of Theorem
\ref{Te2} (stated as Theorem \ref{Te2A}) is proved in Section \ref{S6} of our work.

\section{A priori estimates and  the convergence of  semi-wavefronts} \label{ETW}
As it was suggested in \cite{HTa}, it is convenient to study equation (\ref{twe2an}) together with   
\begin{eqnarray} \label{twe2m} &&
\phi''(t) - c\phi'(t) + g_\beta(\phi(t))(1-
(\phi  * K)(t)) =0,  \end{eqnarray}
where  the continuous piece-wise linear function $g_\beta, \ \beta >1,$ is given by 
\begin{eqnarray}\label{G}
g_{\beta}(u)=\left\{\begin{array}{cc} u,& u \in [0,\beta], \\    \max\{0,2\beta -u\},
& u>  \beta.\end{array}\right. 
\end{eqnarray}
Observe that equation (\ref{twe2m}) has three constant solutions: $\phi(t) \equiv 0, 1, 2\beta$. 
We have the following 
\begin{lem} \label{beA} Assume that $\phi, \ \phi(-\infty) =0,$ is a non-negative, bounded and non-constant solution of 
(\ref{twe2m}). Then $\phi(t) \leq 2\beta$ for all $t \in \R$.  Next, if  either $t_0$ is a point of local maximum for $\phi(t)$ with $\phi(t_0) < 2\beta$ or  $t_0$ is the smallest number such that $\phi(t_0) =2\beta$, then $(\phi * K)(t_0) \leq 1$.  
\end{lem}
\begin{proof} 
On the contrary, suppose that there exists a maximal interval  $(t_0, t_1)$, such that   $\phi(t) > 2\beta = \phi(t_0)$ for all $t \in (t_0, t_1)$. Then  $\phi'(t_*) > 0, \phi(t_*) > 2\beta$ for some $t_* \in (t_0,t_1)$. { It follows from (\ref{twe2m}) and the definition of $g_\beta$ that }
$\phi''(t) = c\phi'(t)$ for all  $\ t \in (t_0,t_1)$.  Hence, $\phi'(t) = \phi'(t_*)e^{c(t-t_*)} >0,$ $t \in (t_0,t_1)$ and therefore $t_1= +\infty$, $\phi(+\infty) = +\infty$,  contradicting the boundedness of  $\phi$.

Finally, if   $t_0$ is a point of local maximum for $\phi(t)$,  then $\phi'(t_0) = 0,$
 $\phi''(t_0) \leq  0$. If, in addition, $\phi(t_0) <2\beta$ then $g_\beta(\phi(t_0)) >0$ and  thus (\ref{twe2m}) 
assures that  $(\phi * K)(t_0) \leq 1$. In the case when $t_0$ is the smallest number such that $\phi(t_0) =2\beta$,
 then clearly there exists a sequence 
 $t_j \to t_0,$ $ t_j < t_0, \ j =1,2, \dots$ such that $\phi'(t_j) >0, \ \phi''(t_j) <0, \ \phi(t_j) < 2\beta$. But then 
$(\phi * K)(t_j) < 1,$ for all $j$ and therefore also $(\phi * K)(t_0) \leq 1$. \qed
\end{proof}
The  following  property of solutions to (\ref{twe2an}) and (\ref{twe2m}) was established in \cite[Lemmas 3.7 and 3.9]{BNPR}:   
\begin{lem} \label{bebe} Assume that $\phi$ is a non-negative, bounded and non-constant solution of 
(\ref{twe2m}) or (\ref{twe2an}). If, in addition, $\phi(t_n) \to 0$ along some sequence $t_n \to -\infty,$ then $\phi(t)$ is  increasing on some interval $(-\infty, \rho], \ \phi(-\infty) = 0,$ and  \
$
\liminf_{t \to +\infty} \phi(t) >0. 
$
\end{lem}
In fact, it is easy to see that each non-trivial non-negative profile should be positive: 
\begin{lem} \label{po} Let a non-negative bounded  $\phi \not\equiv 0$ solve either (\ref{twe2m}) or (\ref{twe2an}) and $c \geq 2$. Then 
$$\phi(t) >0,\quad - \phi'(t)/ \phi(t) > -\lambda(c)
.$$  
If, in addition,  $\phi(-\infty)=0,$
$\phi(t) \leq 1, \ t \in \R,$ then  $\phi'(t) >0$ for all $t \in \R$ and  $\phi(+\infty) = 1$. 
\end{lem}
\begin{proof} First, notice that equation (\ref{twe2m})  with $\beta=+\infty$ coincides with  (\ref{twe2an}), 
so it suffices to consider  equation (\ref{twe2m}) allowing $\beta=+\infty$. 
Suppose that, for some $s$,  solution $\phi$ of (\ref{twe2m}) satisfies  $\phi(s)=0$. Since $\phi(t) \geq 0, \ t \in \R,$ this 
yields $\phi'(s)=0$. Notice that $y=\phi(t)$ is the solution of the following initial value problem 
for a linear second order ordinary differential equation
$$
y''(t) -cy'(t) +a(t)y(t)=0, \quad y(s)= y'(s) =0,
$$ 
where 
$$
a(t) := \left\{\begin{array}{cc} 1 - (\phi * K)(t),&0 \leq\phi(t) \leq \beta, \\  \frac{g_{\beta}(\phi(t))}{\phi(t)} (1-( \phi * K) (t))  ,
& \phi(t)>  \beta,
\end{array}\right.
$$
is a continuous bounded function. 
But then $y(t)  \equiv 0$ due to the uniqueness theorem, a contradiction. 

Suppose now that $\phi$  satisfies (\ref{twe2m}) and $c> 2$. 
Set $${\mathcal N}(\phi)(t):=g_\beta (\phi(t))
(\phi * K)(t)+ \phi(t)- g_\beta (\phi(t)),$$ then ${\mathcal N}(\phi)(t)> 0$ and 
\begin{equation}\label{re}
\phi(t) = \frac{1}{\mu-\lambda}
\int_t^{+\infty}(e^{\lambda (t-s)}- e^{\mu
(t-s)}){\mathcal N}(\phi)(s)ds.
\end{equation}
As a consequence, we have that 
$$
\phi'(t) = \frac{1}{\mu-\lambda}
\int_t^{+\infty}(\lambda e^{\lambda (t-s)}-\mu  e^{\mu
(t-s)}){\mathcal N}(\phi)(s)ds,
$$
and therefore
$$
\phi'(t) -\lambda \phi(t) = - \int_t^{+\infty}e^{\mu
(t-s)}{\mathcal N}(\phi)(s)ds < 0. 
$$ 
If now $c=2$, we find similarly that 
$$
\phi(t) =
\int_t^{+\infty}(s-t)e^{t-s}{\mathcal N}(\phi)(s)ds,\quad \phi'(t) =
\int_t^{+\infty}(s-t-1)e^{t-s}{\mathcal N}(\phi)(s)ds,
$$
and thus also 
$$
\phi'(t) -\phi(t)=  -\int_t^{+\infty}e^{
t-s}{\mathcal N}(\phi)(s)ds < 0. 
$$
Finally, $0 < \phi(t) \leq 1, \ t \in \R,$ implies  that $\phi''(t) - c\phi'(t) \leq 0$.  As a consequence, $ \phi'(s) \geq \phi'(t)e^{c(s-t)}, \  s \leq t$, so that there either  exists  
 a sequence $t_n\to +\infty$ such that $\phi'(t_n) >0$, or  there exists  the leftmost $T_1 \in \R\cup \{+\infty\}$ such that 
$\phi'(t) \leq0$ for all $t \geq T_1$.  In the first case, $\phi'(t) >0, \ t \in \R,$ while in the second case $\phi(t)$ is non-increasing  and $\phi''(t) \leq c\phi'(t) \leq  0,$ for $ \ t \geq T_1$. Since $\phi(t) > 0$,  this can only happen when $\phi(t) \equiv \phi(T_1)$ for $ \ t \geq T_1$.   {But then $( \phi * K) (T_1)=1$, which implies  $\phi(t) =1$ for  $t \geq T_1$ and $K(s) = 0$ a.e. on $\R_+$. }Ê Furthermore, $\phi'(s) >0$ for $s < T_1$. Now, observe that both $\phi(t)$ and $1$ satisfy equation (\ref{re}) and that 
${\mathcal N}(\phi)(t)={\mathcal N}(1)(t)=1$ for all $t \geq T_1$ and ${\mathcal N}(\phi)(t)= \phi(t)
(\phi * K)(t)$ for $t \leq T_1$.  Therefore (\ref{re}) implies that, for $t<T_1$ close to $T_1$, 
$$
0< 1-\phi(t) = \frac{1}{\mu-\lambda}
\int_t^{T_1}(e^{\lambda (t-s)}- e^{\mu
(t-s)})(1-\phi(s)
(\phi * K)(s))ds\leq 
$$
$$
 \frac{1}{\mu-\lambda}
\int_t^{T_1}(e^{\lambda (t-s)}- e^{\mu
(t-s)})ds(1-\phi^2(t)) = (t-T_1)^2(0.5+o(1))(1-\phi(t))(1+\phi(t)), \ t \to T_1-, 
$$
a contradiction. 
\qed
\end{proof}

\begin{lem} \label{li} Let a positive bounded  $\phi$  solve (\ref{twe2m}) 
and there exists the limit 
$\phi(+\infty)$. Then $\phi(+\infty) \in \{1,2 \beta\}.$    If  $\phi(+\infty) =2\beta$ then $\phi(t) \equiv 2\beta$ on some maximal nonempty interval $[T_1, + \infty)$ and $(\phi * K)(T_1) \leq 1$. Furthermore, if $ {2}\beta \int_{-\infty}^0K(s)ds >1$ then 
 $\phi(+\infty) =1$.
\end{lem}
\begin{proof} It follows from Lemmas \ref{beA} and \ref{bebe} that $\phi(+\infty) \in (0, 2\beta]$. In addition, 
if $\phi(+\infty) \not\in \{{1}, 2\beta\},$ then for
$$
r(t): = g_\beta (\phi(t))(1-
(\phi * K)(t)),   
$$
we have that 
$$
\lim_{t \to +\infty}r(t)   =  g_\beta (\phi(+\infty))(1-
\phi(+\infty))   \not= 0.
$$

However,  in this case  the differential equation 
$\phi''(t)- c\phi'(t) + r(t) =0$ does not have any 
convergent solution on $\R_+$. Indeed,  we have that  
$$
|\phi'(t)| = \left|\phi'(s) + c(\phi(t)-\phi(s)) -\int_s^tr(u)du\right| \to +\infty \ \mbox{as} \ t \to +\infty. 
$$
Finally, assume that $\phi(+\infty) = 2\beta$, then  there exists $T_1 \in \R$ such that  $r(t)\leq 0$ for $t\in [T_1, \infty)$  and thus
$\phi''(t) -c\phi'(t) \geq 0$  for $t \geq T_1$. As a consequence, $\phi'(t) \geq \phi'(s)e^{c(t-s)}$ for $t\geq s \geq T_1$. If $\phi'(s) >0$ for some $s \geq T_1$, we obtain a contradiction:
$\phi'(+\infty) = +\infty$. Therefore we have to analyse the case when $\phi'(s) = 0$ for all $s \geq T_1$
(we can assume that $T_1$ is the smallest number with such a property).  By Lemma \ref{beA}, 
$${2}\beta \int_{-\infty}^0K(s)ds =\int_{-\infty}^0\phi(T_1-s)K(s)ds \leq (\phi * K)(T_1) \leq 1,$$ 
which proves the last statements of the lemma. \qed
\end{proof}
\begin{remark}\label{ir}
Suppose that  $\int_{-\infty}^0K(s)ds > 0$. Then we can choose $\beta$ large enough to meet the inequality  $ {2}\beta \int_{-\infty}^0K(s)ds >1$. Hence, if $\int_{-\infty}^0K(s)ds > 0$ and there exists $\phi(+\infty)$,   we can assume that   $\phi(+\infty) =1$.
\end{remark}

\vspace{3mm}

Now, the change of variables 
{\begin{equation}\label{x}
\phi(t) = e^{-x(t)}, \;\;i.e.\;\; x(t) = -\ln \phi(t),
\end{equation}}
 transforms equation  (\ref{twe2an})  into  
$$
x''(t) - cx'(t) - (x'(t))^2 +\left(\int_{-\infty}^{+\infty}e^{-x(t-s)}K(s)ds-1\right)=0, \ t \in \R. 
$$
We will also consider the  family of equations 
$$
x''(t) - cx'(t) - (x'(t))^2 +h_\beta(x(t))\left(\int_{-\infty}^{+\infty}e^{-x(s)}K(t-s)ds-1\right)=0, \ t \in \R, 
$$
where non-decreasing continuous function $h_\beta: \R \to [0, +\infty)$,\ $\beta >1,$ is defined by 
$$h_\beta(x)=\left\{\begin{array}{cc} 1,& x  \geq - \ln \beta, \\  \max\{0,2\beta e^x -1\},
& x\leq  -\ln \beta.\end{array}\right. $$  
For $c \geq 2$,  we will consider a strictly increasing function $f: [-1, \infty) \to \R$,
$$
f(s) := \frac{2s}{c+\sqrt{c^2+4s}}.
$$
\begin{lem} \label{ogran} For each  $c\geq 2$ and $K \in \mathcal{K}:= \{K\in L^1(\R,\R_+): |K|_1=1, K\geq 0\}$ there exists  $U(c,K) \geq 1$ depending only on $c$ and $K$ such that the following holds: if $\phi(t)$, $\phi(-\infty) =0$,  is a positive bounded  solution of  equation (\ref{twe2m}) with $\beta > U(c,K)$,  then 
\begin{equation}\label{iu}
0< \phi(t) \leq U(c,K), \ t \in \R
\end{equation}
(i.e. the set of all semi-wavefronts to (\ref{twe2m})  is uniformly bounded by a constant which does not depend on a particular semi-wavefront). Moreover, given a fixed pair $(c_0,K_0) \in [2,+\infty) \times \mathcal{K}$,  we can assume that  the map 
$U: [2,+\infty) \times \mathcal{K}  \to (0,+\infty)$ is locally continuous at $(c_0,K_0)$. 
\end{lem}
\begin{proof} 
First, we take $U(c,K) \geq 1$ defined by one of the following non-exclusive formulas: 
\begin{itemize}
\item  if $\int_{0}^{+\infty}K(s)ds >0$, then $U(c,K)= \left(\int_{0}^{+\infty}e^{s f(-1)}K(s)ds\right)^{-1}$;
\item if $\int_{0}^{+\infty}K(s)ds < 0.001$, then $U(c,K)= 2\exp{(\lambda (r+\sigma))},$ where  $\lambda= \lambda(c)$ is defined by (\ref{E}) and $r = r(K) \in \N,\ \sigma=\sigma(c)>0$ are chosen in such a way that $$\int^{0}_{-r}K(s)ds> 0.99, \quad 
2c\frac{e^{\lambda \sigma} -1}{e^{c \sigma} -1} < 0.01.$$ 
\end{itemize}
Obviously,  such $U: [2,+\infty) \times \mathcal{K}  \to (0,+\infty)$ is locally continuous at each $(c_0,K_0)$.  For example, $U(c,K)$ can be considered as a constant (hence, continuous) function in some small neighborhood of $(c_0,K_0) \in [2,+\infty) \times  \mathcal{K}$ satisfying $\int_{0}^{+\infty}K_0(s)ds =0$.

Clearly, if $\phi(t) \in (0,1]$ for all $t \in \R$, then inequality (\ref{iu}) is true because of  $U(c,K)\geq 1$. In particular, this happens if 
the profile  $\phi(t)$ is {nondecreasing}  and  $2\beta\int_{-\infty}^0K(s)ds > 1$, see Remark \ref{ir}.

Thus let us suppose that $\phi(t_0) > 1$ at some point $t_0$.  Then at least one of the following three possibilities can occur: 

\noindent \underline{Situation I}. Solution $\phi(t)$ is {nondecreasing}  and  $\int_{-\infty}^0K(s)ds = 0$ (so that  $\int_{0}^{+\infty}K(s)ds =1$).  In such a case, by Lemma \ref{li}, there is  some finite $T_1$ such that 
$\phi(+\infty)=\phi(T_1)=2\beta$ and $(\phi * K)(T_1) \leq 1$. {For $x$ defined by (\ref{x})}, we have  $ x'(t) = -\phi'(t)/\phi(t) \geq -\lambda(c) = f(-1)$  for all $t\in \R$ 
and 
$$
\int_{0}^{+\infty}e^{-x(T_1-s)}K(s)ds \leq (\phi * K)(T_1)= \int_{\R}e^{-x(T_1-s)}K(s)ds \leq 1.
$$
Now, set $m:= \min_{s\in \R} x(s)$ and observe that $x(t) = x(T_1)-\int^{T_1}_t x'(s)ds \leq m +f(-1)(t-T_1)$ for $t \leq T_1$.  Thus 
$$
\int_{0}^{+\infty}e^{-m+s f(-1)}K(s)ds \leq \int_{0}^{+\infty}e^{{-}x(T_1-s)}K(s)ds \leq 1
$$
and therefore 
$$
\frac{1}{2\beta} = \frac{1}{\phi(T_1)} = e^m \geq \int_{0}^{+\infty}e^{s f(-1)}K(s)ds. 
$$
Thus we can take 
$$
2\beta= \phi(T_1)  \leq \left(\int_{0}^{+\infty}e^{s f(-1)}K(s)ds\right)^{-1} = U(c,K).
$$
{The latter shows that  Situation I cannot occur if $2\beta > U(c,K)$.}

\vspace{2mm}
 
\noindent  \underline{Situation II}. Solution $\phi(t)$ is not {nondecreasing} 
 { and  $\int_{0}^{+\infty}K(s)ds>0.$} 
Then we can repeat the above arguments to conclude that, for the local maxima $\phi(t_j) >1$ of $\phi$ we have that 
 $$
\sup_{t \in \R}\phi(t) =\sup_j \phi(t_j)  \leq \left(\int_{0}^{+\infty}e^{s f(-1)}K(s)ds\right)^{-1}= U(c,K). 
 $$
\underline{Situation III}. 
{Solution $\phi(t)$ is not {nondecreasing} 
  and $\int_{0}^{+\infty}K(s)ds=0$.} 
Suppose, on the contrary,  that $\phi(t_0) > U(c,K) = 2\exp{(\lambda (r+\sigma))}$ for some $t_0$.  Then $\phi(t) \geq 2$ on some maximal closed interval $[a,b] \ni t_0$.  We claim that $b-a \geq r+\sigma$.  Indeed, otherwise, since $\phi'(t) \leq \lambda \phi(t), $ $\phi(a) =2,$ $ \phi'(a) \geq 0, \ $ $ t_0-a < b- a,$  we get the following contradiction 
$$
\phi(t_0) \leq \phi(a) \exp(\lambda (t_0-a)) < 2\exp(\lambda (b-a)) \leq   2\exp(\lambda (r+\sigma)). 
$$
In consequence, 
$$
(\phi * K)(t)  \geq  \int_{-r}^0 \phi(t-s)K(s)ds \geq 1.98, \quad t \in [a, a+\sigma], 
$$
so that $(1-(\phi * K)(t)) \leq -0.98, \ t \in [a,a+\sigma],$ and $\phi''(t) -c \phi'(t) > 0,$ $\ t \in [a,a+\sigma]$.  In particular, $\phi'(t) > e^{c(t-a)}\phi'(a)$ for all $a  <t \leq a+\sigma$ and thus 
$$
2\exp(\lambda (t-a))  \geq \phi(t) > 2+ \frac 1c \{e^{c(t-a)}-1\}\phi'(a), \quad a  <t \leq a+\sigma. 
$$
Therefore 
$$
2e^{\lambda \sigma} > 2+ \frac 1c \{e^{c\sigma}-1\}\phi'(a), 
$$
so that $0 \leq \phi'(a) <0.01$.  Next, let $[a_-, b_+] \supseteq [a,b]$ be the maximal interval 
where $\phi(t) \geq 1.1$. Then, for all $t \in [a_-, a + \sigma]$, we have $\phi''(t)-c\phi'(t) >0$ since
$$
(\phi * K)(t) \geq  \int_{-r}^0 \phi(t-s)K(s)ds \geq 0.99\cdot 1.1 > 1, \ t \in [a_-, a+\sigma]. 
$$
But then 
$$
\phi'(t) < \phi'(a)e^{c(t-a)} < 0.01e^{c(t-a)}, \ t \in [a_-,a);
$$
$$
\phi(t) > 2- \frac{0.01}{c}\{1-e^{c(t-a)}\} > 2-  \frac{0.01}{c} > 1.9,  \quad t \in [a_-,a), 
$$
a contradiction (since $\phi(a_-) =1.1$). 
\qed
\end{proof}
\begin{cor}
\label{twoe} Assume that $c\geq 2$ and $K$ are fixed.  Then, for each sufficiently large $\beta >1$,  equations (\ref{twe2m}) and (\ref{twe2an}) share the same set of semi-wavefronts propagating at the velocity $c$. 
\end{cor} 
\begin{proof} Due to Lemma \ref{ogran} and the definition of $g_{\beta}(u)$, it suffices to take
$\beta > U(c,K)$. \qed
\end{proof}
A stronger {\it a priori} estimate is based on the following  assertion: 
\begin{lem}\label{L20}
Let $y$ be a bounded solution of the  boundary problem 
$$
y'(t)-cy(t) -y^2(t) + g(t) =0,\  y(b)=0,  \ t \in (a,b], 
$$
where $c \geq 2$ and  {a continuous  function } Ê$g$ satisfies $$-1 < A:= \inf_{s\in (a,b]} g(s). $$ 
Set $B:=  \sup_{s\in (a,b]} g(s).$
If there exists $\omega:= \min_{s\in (a,b]} y(s)$ and $A<0$,  then 
$
\omega \geq {f(A)}. 
$
Similarly, if there exists $\gamma:= \max_{s\in (a,b]} y(s),$ then 
$
\gamma \leq f(B). 
$
\end{lem}
\begin{proof} The above statements were proved  in \cite[Lemma 20]{HTa} under additional condition $y(a) =0,$ but without assuming the existence of the global extrema of $y$ on $(a,b]$. It is easy to check that the latter condition (which is obviously weaker than $y(a)=0$)  is sufficient to {repeat} all the arguments in the proof of  \cite[Lemma 20]{HTa}. \qed \end{proof}
The next two results can be considered as  improvements of Lemma \ref{ogran}. 
\begin{lem} \label{ogrand} Let $c\geq 2$ and $\phi(t), \ \phi(-\infty) =0,$ be a bounded positive 
 solution of  equation 
(\ref{twe2an}).  Set 
$  \ \rho (u) = f(e^{-u}-1)$, $x(t) = -\ln \phi(t)$ and 
{
$$
m = \liminf_{t \to +\infty} 
x(t), \quad M= \limsup_{t \to +\infty} 
 x(t). 
$$}
Then  
$$
 \int_0^{+\infty} e^{\rho(m)s}K(s)ds + \int_{-\infty}^0e^{\rho(M)s}K(s)ds \geq e^{M}, 
$$
$$
 \int_0^{+\infty} e^{\rho(M)s}K(s)ds + \int_{-\infty}^0e^{\rho(m)s}K(s)ds \leq e^{m}.
 $$
 \end{lem}
 
 Observe that the integrals in the statement of Lemma \ref{ogrand} (and in Lemma \ref{L9} below as well) can be infinite (i.e. equal to $+\infty$).  
\begin{proof}  By Lemma \ref{bebe}, 
 the wavefront profile $\phi(t)$ is increasing on some maximal interval $(-\infty, Q_0)$ and 
$\liminf_{t \to + \infty} \phi(t) >0$.  Moreover, if $\phi(t)$ is eventually monotone and $\beta$ is sufficiently large 
then $\phi(+\infty)=1$ by Lemmas \ref{li} and \ref{ogran}.  In such a case, $M=m=0$, which proves the lemma. 
Hence,  we may assume that $\phi(t)$ is not eventually monotone.  Set $y(t) = x'(t)$, since $x(t)$ is neither eventually monotone there exists some $s > Q_0$ such that $y(s) >0$. Moreover, 
it is clear that for each such $s$ we can find some finite $a<s< b$ such that  $y(s)> 0=y(b)= y(a)$.  Then  Lemmas \ref{L20} and \ref{ogran} assure that 
$$
 y(s) \leq f\left(\max_{t \in [a,b]} \int_\R e^{-x(t-u)}K(u)du -1\right) \leq f(U(c,K)-1). 
$$
In particular, this means that  $\sup_{s \in \R} y(s)$ is a finite number. We claim that 
\begin{equation}\label{1y}
\limsup_{s \to +\infty} y(s) \leq \rho(m).
\end{equation}
Indeed, let $s_j \to +\infty$ be such that $y(s_j) \to \limsup_{s \to +\infty} y(s)$. Then for appropriately chosen 
sequences $a_j < s_j < b_j$, $a_j \to +\infty$,  we have that 
$$
 y(s_j) \leq f\left(\max_{t \in [a_j,b_j]} \int_\R e^{-x(t-u)}K(u)du -1\right). 
$$
Next, by Lemma \ref{ogran}, for every small $\epsilon >0$ there exists $T= T(\epsilon, c, K), \ T(0^+, c, K)=+\infty,$ such that 
\begin{equation} \label{TT}
 \hspace{-7mm} \int_{-T}^{T}K(u)du > 1-\epsilon, \ \int_{-\infty}^{-T} e^{-x(t-u)}K(u)du + \int^{+\infty}_{T} e^{-x(t-u)}K(u)du < \epsilon, \ t \in \R. 
\end{equation}
Consequently, 
$$
 y(s_j) \leq f(\max_{t \in [a_j-T,b_j+T]}e^{-x(t)}  -1+\epsilon).
$$
Taking into account that $\liminf_{j \to \infty}\min_{t \in [a_j-T,b_j+T]}x(t) \geq m$, we conclude that 
$$
\limsup_{s \to +\infty} y(s) \leq f(e^{-m} -1+\epsilon). 
$$
Letting $\epsilon \to 0^+$ in the last inequality, we obtain (\ref{1y}).  

Next, Lemma \ref{po} implies that $y(s) > f(-1)\geq -1 >-c$ for all $s \in\R$. Since $x(t)$ is not eventually monotone, there exist 
sequences $d_j < \varsigma_j < e_j$, $d_j \to +\infty,$ such that $\min_{s\in [d_j,e_j]}y(s) =  y(\varsigma_j)< 0=y(d_j)= y(e_j)$ and 
$y(\varsigma_j) \to \liminf_{s \to +\infty} y(s)$. Set $g(t) =  \int_\R e^{-x(t-u)}K(u)du -1$.  Since $y'(\varsigma_j)=0$, we obtain that 
$$-1< \min_{s \in [d_j,e_j]}g(s) \leq g(\varsigma_j) = y^2(\varsigma_j)+cy(\varsigma_j)<0.$$ 
Therefore Lemma \ref{L20} can be applied yielding   
$$
 y(\varsigma_j) \geq f\left(\min_{t \in [d_j,e_j]} \int_\R e^{-x(t-u)}K(u)du -1\right). 
$$
From this estimation, arguing as above, we deduce that 
$$
\liminf_{s \to +\infty} y(s) \geq \rho(M) > f(-1).
$$
Next, let $t_j \to +\infty$ be a sequence of local maximum points of $x(t)$ such that $x(t_j) \to M$ as $j \to +\infty$.  
With $T$ and $\epsilon$ as in (\ref{TT}) and  for sufficiently large $j$, we find that
 $$m'_j:= \min_{s \in [t_j,t_j+T]}x'(s) > \rho(M)-\epsilon, \ M'_j:= \max_{s \in [t_j,t_j+T]}x'(s) < \rho(m)+\epsilon,$$
 $$m_j:= \min_{[t_j-T,t_j+T]}x(s) > m -\epsilon, \ M_j:= \max_{[t_j-T,t_j+T]}x(s) < M+\epsilon, $$
$$
\epsilon + \int_{-T}^Te^{-x(t_j-s)}K(s) > \int_{\R}e^{-x(t_j-s)}K(s)ds \geq 1, 
$$
$$
x(t) \geq x(t_j) + m'_j(t-t_j) \geq x(t_j) + (\rho(M)-\epsilon)(t-t_j), \  t\in[t_j, t_j +T], 
$$
$$
 x(t)  \geq x(t_j) +  M'_j(t-t_j) \geq x(t_j) + (\rho(m)+\epsilon)(t-t_j),\ t \in [t_j-T, t_j]. 
$$
Therefore, for each subset $[A,B] \subset [-T,T], \ A \leq 0\leq B,$ we obtain
$$
\epsilon + \int_{-T}^0e^{-x(t_j-s)}K(s)ds  + \int_0^{T}e^{-x(t_j-s)}K(s)ds>  1, 
$$
$$
\epsilon + \int^{A}_{-T}e^{-m+\epsilon}K(s)ds + \int_{A}^0e^{-x(t_j)+ (\rho(M)-\epsilon)s}K(s)ds  + $$
$$
\int_0^{B}e^{-x(t_j)+(\rho(m)+\epsilon)s)}K(s)ds+  \int_B^{T}e^{-m+\epsilon}K(s)ds>  1. 
$$
Taking limit in the last inequality when $\epsilon \to 0$ (so that $T \to +\infty$), $j \to +\infty$, we obtain that  
\begin{eqnarray} \label{ABe} 
e^{-m} \left(\int^{A}_{-\infty}K(s)ds +\int_{B}^{+\infty}K(s)ds\right) +  e^{-M}\left(\int_{A}^0e^{\rho(M)s}K(s)ds  + 
\int_0^{B}e^{\rho(m)s}K(s)ds\right)\geq  1. 
\end{eqnarray}
This relation is valid for each $-\infty \leq A \leq 0 \leq B \leq +\infty$ and if $A, B$ are infinite, we get the first inequality of the lemma. Clearly, the second inequality can be deduced in a similar way from  
\begin{eqnarray} \label{A'B'e} 
e^{-M} \left(\int^{A'}_{-\infty}K(s)ds +\int_{B'}^{+\infty}K(s)ds\right) +  e^{-m}\left(\int_{A'}^0e^{\rho(m)s}K(s)ds  + 
\int_0^{B'}e^{\rho(M)s}K(s)ds\right)\leq  1, 
\end{eqnarray} 
where $A', B'$ are arbitrary real numbers satisfying $-\infty \leq A' \leq 0 \leq B' \leq +\infty$. 
\qed
\end{proof}
\begin{lem}\label{L9} Let $\phi(t)$ be a semi-wavefront to equation 
(\ref{twe2an}) propagating with the speed $c\geq 2$. Set $$
p = \liminf_{t \to +\infty} 
\phi(t), \ P= \limsup_{t \to +\infty}\phi(t),\ \alpha_+:= \frac{1}{c}\int_{-\infty}^0|s|K(s)ds, \ \alpha_-:= \frac{1}{c}\int^{+\infty}_0sK(s)ds. 
$$
Then $0< p \leq 1 \leq P$ and 
\begin{equation}\label{cine}
p+\alpha_+ P(1-p) +\alpha_- P(P-1) \geq 1,
\end{equation}
\begin{equation}\label{cine2}
 P-\alpha_+ P(P-1) -\alpha_- P(1-p) \leq 1. 
\end{equation}
\end{lem} 
\begin{proof} Taking $A=B=A'=B'=0$ in (\ref{ABe}), (\ref{A'B'e}) we find immediately that   $0< p \leq 1 \leq P$. In the case when $p=P$, Lemma \ref{li} and Corollary \ref{twoe} imply that $p=P=1$ and that proves the lemma. 
If $p <P$, $\phi(t)$ oscillates between $p$ and $P$. Therefore $\phi'(t)$ is oscillating around $0$  and there exist finite limits
$$
d = \liminf_{t \to +\infty} 
\phi' (t) \leq 0 \leq D: = \limsup_{t \to +\infty}\phi'(t). 
$$
We claim that 
$$
-\frac 1c P(P-1) \leq d \leq D \leq \frac 1c P(1-p). 
$$
Indeed, let $t_j \to +\infty$ be such that $0> \phi'(t_j) \to d, \ \phi''(t_j)=0$. Then 
$$
-\phi'(t_j) = \frac 1c \phi(t_j) (\phi * K(t_j)-1). 
$$ 
For an arbitrary $\epsilon >0$, we fix $T$ sufficiently large 
to have 
$$
\int_{-\infty}^{-T}\phi(t-s)K(s)ds +  \int^{+\infty}_{T}\phi(t-s)K(s)ds < \epsilon, \quad t \in \R.
$$
Then 
$$
-\phi'(t_j) \leq \frac 1c \phi(t_j) \left(\epsilon + \int_{-T}^{T}\phi(t_j-s)K(s)ds-1\right) $$
$$
 \leq \frac 1c \sup_{u \geq 0.5 t_j}\phi(u) \left(\epsilon + \sup_{u \geq 0.5 t_j}\phi(u) \int_{-T}^{T}K(s)ds-1\right).
$$  
After taking limit as $j \to +\infty,\  \epsilon \to 0^+$ (so that $T \to +\infty$), we get  one of the required relations: $-d \leq P(P-1)/c$.  The second inequality can be proved similarly. 

Next, consider the sequence $\{s_j\}$ of  local maximum points such that 
$\phi(s_j) \to P,$ $s_j \to +\infty$.  We can suppose that $s_j$ is large enough to have 
$$
\min_{s \in [s_j-T,s_j+T]}\phi'(s) \geq d-\epsilon, \quad \max_{s \in [s_j-T,s_j+T]}\phi'(s) \leq D+\epsilon. 
$$
Then 
$$
\phi(s_j-s) \geq \phi(s_j) - (D+\epsilon)s, \ \mbox{for} \ s \in [0,T], \quad  \phi(s_j-s) \geq \phi(s_j) - (d-\epsilon)s,\  \mbox{for} \ s \in [-T,0], 
$$
and therefore 
$$
1 \geq \int_{-\infty}^{+\infty}\phi(s_j-s)K(s)ds \geq \int_{-T}^{T}\phi(s_j-s)K(s)ds \geq $$
$$
\phi(s_j)\int_{-T}^{T}K(s)ds - (D+\epsilon) \int_{0}^TsK(s)ds+ (d-\epsilon)\int_{-T}^{0}|s|K(s)ds. 
$$
Finally,  letting $\epsilon \to 0^+$ (hence, $T \to +\infty$), $s_j \to +\infty$, we get the required inequality 
$$
1 \geq P - Dc\alpha_- +dc\alpha_+ \geq P -\alpha_- P(1-p) -\alpha_+ P(P-1). 
$$
The proof of inequality (\ref{cine}) is  completely analogous and therefore it is omitted here. 
\qed
\end{proof}
\begin{remark}
Inequality (\ref{cine2}) has a simple geometric interpretation. Indeed, consider the 
following function 
$$
\tilde \phi_-(-s) := \left\{\begin{array}{cc} P -P(1-p)s/c,&  s \geq 0, \\  P +P(P-1)s/c ,
& s\leq 0,
\end{array}\right.
$$
then  inequality (\ref{cine2}) can be written as $\tilde \Theta(p,P):= (K * \tilde\phi_-)(0) \leq 1$.  A serious drawback of 
the obtained estimate is that $\tilde \Theta(p,P)$ can be negative and therefore the relation $\tilde \Theta(p,P)\leq 1$ is not very useful. We can avoid this imperfection by introducing the function 
$
\phi_-(-s, p, P) := \max\{p, \tilde \phi_-(-s)\}.  
$
Arguing as in the proof of Lemma \ref{L9}, we can find that $p$ and $P$ should satisfy the following improved inequality
$
\Theta(p,P) =(K * \phi_-)(0) \leq 1.
$
Obviously,  continuous function $\Theta(p,P)$ is non-negative for all $0\leq p\leq P$. 
\end{remark}

\section{Existence of semi-wavefronts for $c\geq 2$.} \label{ESW}
In this section, we are going to prove Theorem \ref{Te1}.  It should be observed that the necessity of the condition 
$c\geq 2$ can be easily obtained from the analysis of the asymptotic behaviour of eventual semi-wavefront $\phi$ at $-\infty$ (if $c<2$ then $\phi(t)$ oscillates around $0$ at $-\infty$). Thus we have to prove only the sufficiency of the condition $c \geq 2$ for the existence of semi-wavefronts. 

First, consider    $$r(\phi)(t):= b\phi(t) + g_{\beta}(\phi(t))(1-(\phi * K)(t)),$$ where  $g_{\beta}(u)$ is defined by (\ref{G}), $\beta$ is as in Corollary \ref{twoe}, and $b>1+2\beta$.   In view of Corollary \ref{twoe}, it suffices to establish that the equation  
\begin{equation}
\label{twe2mm} 
\phi''(t) - c\phi'(t)  - b \phi(t) + r(\phi)(t)=0
\end{equation}
has a semi-wavefront.
Observe that if  a continuous function $\psi(t), \ 0 \leq \psi(t)\leq 2\beta,$ satisfies $0\leq \psi(s) \leq \beta$ at some point $s \in \R$, then  
\begin{equation}\label{Ar}
r(\psi)(s)=  \psi(s)(b +1-
(\psi * K)(s)) \geq 0. 
\end{equation}
Now, if   $\beta\leq \psi(s) \leq 2\beta$, then  $$
r(\psi)(s)=  b
\psi(s)+  (2\beta -\psi(s))(1-(\psi * K)(s))=
$$
\begin{equation} \label{Br}
2\beta(1-(\psi * K)(s))+ \psi(s) (b-1+(\psi * K)(s))> \beta.  
\end{equation}

Next, we consider the non-delayed KPP-Fisher equation  $u_t= u_{xx} +g_{\beta}(u)$.  The profiles $\phi$
of the traveling fronts $u(x,t)= \phi (x+ct)$ for this equation satisfy
\begin{equation}\label{kppm}
\phi''(t) - c\phi'(t)  + g_{\beta}(\phi(t))=0, \ c \geq 2.
\end{equation}
Recall that  $0< \lambda \leq \mu$ denote eigenvalues of equation (\ref{kppm}) linearized around $0$ (i.e. 
$\chi(\lambda) = \chi(\mu) =0$ where $\chi(z) := z^2-cz+1$). In the sequel,  $\phi_+(t)$ will denote the unique monotone front to (\ref{kppm})  normalised (cf. \cite[Theorem 6]{GT})  by the condition $$ \phi_+(t):= (-t)^je^{\lambda t} (1+o(1)), \ t \to -\infty, \ j \in \{0,1\}.$$
Let us note here that  $\phi_+(t)$ for all $t$ such that $\phi_+(t) < \beta,$ satisfies the linear differential equation 
$$
\phi''(t) - c\phi'(t)  + \phi(t)=0.
$$
In particular, if $c >2$  then there exists (see e.g. \cite[Theorem 6]{GT}) $C \geq 0$ such that  
\begin{equation} \label{1mo}
 \phi_+(t):= e^{\lambda t}  - C e^{\mu t}, \ t \leq \phi_+^{-1}(\beta).   \end{equation}
 Let $z_1< 0< z_2$ be the  roots of the equation  $z^2 -cz-b =0$. 
Set $z_{12} =z_2-z_1>0$  and consider the  integral operator $A_m$ depending on $b$ and defined by
$$(A_m\phi)(t) = \frac{1}{z_{12}}\left\{\int_{-\infty}^te^{z_1
(t-s)}r(\phi(s))ds + \int_t^{+\infty}e^{z_2
(t-s)}r(\phi(s))ds \right\}.
$$
\begin{lem}  \label{usl} Assume that  $b > 1+2\beta $ and let  $0\leq \phi(t) \leq \phi_+(t)$, then 
$$0 \leq (A_m\phi)(t) \leq \phi_+(t).$$
\end{lem}
\begin{proof} The lower estimate is obvious since $0\leq \phi(t) \leq \phi_+(t) \leq 2\beta$ and therefore
$r(\phi(t)) \geq 0$ in view of (\ref{Ar}) and  (\ref{Br}). Now, since  $\phi(t) \leq \phi_+(t)$ and $bu+g_{\beta}(u)$ is an increasing function, we find that 
$$
r(\phi(s)) \leq b\phi(t) + g_{\beta}(\phi(t)) \leq b\phi_+(t) + g_{\beta}(\phi_+(t)) =:R(\phi_+(t)). 
$$
Thus 
$$
(A_m\phi)(t) \leq  \frac{1}{z_{12}}\left\{\int_{-\infty}^te^{z_1
(t-s)}R(\phi_+(s))ds + \int_t^{+\infty}e^{z_2
(t-s)}R(\phi_+(s))ds \right\}= \phi_+(t),
$$
and the lemma is proved. \qed
\end{proof}
Lemma \ref{usl} says that $\phi_+(t)$ is an {\it upper} solution for (\ref{twe2mm}), cf. \cite{wz}. Still,  we need to find a {\it lower} solution.  Here, assuming that $c >2$  and that $K$ has a compact support we will use the following  well known  ansatz (see e.g. \cite{wz})  
$$
\phi_-(t)= \max\{0,e^{\lambda t} (1- Me^{\epsilon t})\},
$$
where $\epsilon \in (0, \lambda)$  and $M \gg 1$  is chosen in 
such a way that $- \chi(\lambda+\epsilon)   >(L/M) \int_{-\infty}^{\infty} e^{-\epsilon s} K(s) ds$ (here  $L : = \sup_{t \in \R} \phi_+(t)e^{-\epsilon t}$),  $\lambda +\epsilon < \mu$, and
$$0< \phi_-(t) < \phi_+(t) < e^{\epsilon t}< 1, \quad t < T_c, \ \mbox{where} \ \phi_-(T_c) =0.$$
The above inequality $\phi_-(t) < \phi_+(t)$ is possible due to  representation (\ref{1mo}). We note also  that $(\phi_+*K)(t) \leq L e^{\epsilon t}\int_\R e^{-\epsilon s}K(s)ds$. 
\begin{lem}  \label{nre} Assume that {$c>2$}, { $K$ has a  compact support}, $b > 2\beta +2$. Then the inequality  $\phi_-(t) \leq \phi(t) \leq \phi_+(t), \ t \in \R,$ implies that 
\begin{equation}\label{ul2}
\phi_-(t) \leq (A_m\phi)(t) \leq \phi_+(t), \quad t \in \R. 
\end{equation}
\end{lem}
\begin{proof} 
Due to Lemma \ref{usl}, it suffices to prove the first inequality in (\ref{ul2}) for $t \leq T_c$. 
Since $0< \phi(t) < 1 < \beta, \ t \leq T_c$, we have, for $t\leq T_c $, that 
$$
(A_m\phi)(t) \geq   \frac{1}{z_{12}}\left\{\int_{-\infty}^te^{z_1
(t-s)}r(\phi(s))ds + \int_t^{T_c}e^{z_2
(t-s)}r(\phi(s))ds \right\}=
$$
$$
\frac{1}{z_{12}}\left\{\int_{-\infty}^te^{z_1
(t-s)} \phi(s)(b +1-
(\phi * K)(s)) ds + \int_t^{T_c}e^{z_2
(t-s)}\phi(s)(b +1-
(\phi * K)(s)) ds \right\}\geq 
$$
$$ \frac{1}{z_{12}}\left\{\int_{-\infty}^te^{z_1
(t-s)} \Gamma(s) ds + \int_t^{T_c}e^{z_2
(t-s)}\Gamma(s) ds \right\}= 
$$
$$ \frac{1}{z_{12}}\left\{\int_{-\infty}^te^{z_1
(t-s)} \Gamma(s) ds + \int_t^{+\infty}e^{z_2
(t-s)}\Gamma(s) ds \right\} =: Q(t),
$$
where $\Gamma(s): = \phi_-(s)(b +1-
(\phi_+ * K)(s))$. 

In order to  estimate  $Q(t)$, we first find, for  $t \leq T_c$, that
$$
\phi_-''(t) - c\phi_-'(t)  -b\phi_-(t) + b\phi_-(t) + \phi_-(t)(1- (\phi_{+} * K)(t)) =
$$
$$
- \chi(\lambda+\epsilon) M e^{(\lambda +\epsilon)t} - (\phi_+ * K)(t)e^{\lambda t} (1- Me^{\epsilon t})\geq 
- \chi(\lambda+\epsilon) M e^{(\lambda +\epsilon)t} - e^{\epsilon t }e^{\lambda t} L { \int_{-\infty}^{\infty} e^{-\epsilon s} K(s) ds}= $$
$$ M e^{(\lambda +\epsilon)t}\left(
- \chi(\lambda+\epsilon)   - {\frac L M \int_{-\infty}^{\infty} e^{-\epsilon s} K(s) ds}\right) >0. 
$$
But then, rewriting the latter differential inequality in the equivalent  integral form  (see e.g.  \cite{ma} or \cite[Lemma 18]{TPT}) and using  the fact that 
$$\Delta\phi'_-|_{T_c} := \phi'_-(T_c^+) - \phi'_-(T_c^-) = - \phi'_-(T_c^-) >0, $$
we can conclude that $Q(t) \geq \phi_-(t), \ t \in \R$.  Hence, $(A_m\phi)(t) \geq  \phi_-(t), \ t \in \R$. \qed
\end{proof}

Next, with each  vector $\frak{m} = (\mu_1,\mu_2)$ we will associate the following Banach
spaces:
\[
C_{\frak{m}} = \{y \in C(\R, \R): |y|_{\frak{m}}: = \sup_{s \leq
0}e^{-\mu_2 s} |y(s)|+  \sup_{s \geq
0}e^{-\mu_1s} |y(s)| < + \infty  \},
\]
\[
C^1_{\frak{m}} = \{y \in C_{\frak{m}}: y' \in C_{\frak{m}}, \
|y|_{1,\frak{m}}: =  |y|_{\frak{m}} +  |y' |_{\frak{m}} < +\infty
\}. \]
\begin{remark}
Observe that $C_{\frak{m}}=C^0(\mu_2, \mu_1)$, $C^1_{\frak{m}}=C^1(\mu_2, \mu_1)$  in the notation of \cite[p. 185]{HL}.
\end{remark}
It is clear that,  in order to establish the existence of semi-wavefronts to equation (\ref{twe2mm}), it suffices to prove that 
the equation $A_m\phi=\phi$ has at least one  solution  from the set 
$$\frak{K} = \{x \in C_{\frak{m}}: \phi_-(t) \leq x(t) \leq  \phi_+(t), \ t \in \R\},$$ where 
$ {\frak{m}}= (\rho, \lambda/2)$ for some fixed $\rho >0$.  Observe that 
the convergence $x_n \to x$ in $\frak{K}$ is equivalent to the
uniform convergence  on compact subsets of
$\R$.
\begin{lem} Let $c >2$. Then $\frak{K}$ is a closed, bounded, convex subset of $C_{\frak{m}}$
and $A_m:\frak{K} \to \frak{K}$ is completely continuous.
\end{lem}
\begin{proof} By the previous lemma, $A_m(\frak{K}) \subset \frak{K}$. It is also obvious that $\frak{K}$ is a  closed, bounded, convex subset of $C_{\frak{m}}$. Since 
\begin{equation}\label{AA}
|x(t)|+ |(A_mx)'(t)|\leq 2\beta(1+z_{12}),\  \mbox{for all} \   x \in \frak{K},  
\end{equation}
due to the Ascoli-Arzel${\rm
\grave{a}}$ theorem  $A_m(\frak{K})$ is relatively compact in $\frak{K}$ . 
Next, by Lebesgue's dominated convergence theorem, if $x_j\to  
x_0$ in $\frak{K}$ then 
$(A_mx_j )(t) \to  
(A_mx_0 )(t)$ at every $t \in  
\R$. The precompactness of 
$\{A_mx_j \} \subset \frak{K}$ assures that, in 
fact, $A_mx_j\to A_mx_0$ in $\frak{K}$. Hence, the map $A_m: \frak{K}\to \frak{K}$ is completely  continuous. \qed
\end{proof}
The final steps of the proof of Theorem \ref{Te1} are contained  in the following proposition.
\begin{theorem} \label{34} Assume that  $c\geq 2$. Then the integral equation
$A_m\phi=\phi$  has at least one  positive bounded solution in $\frak{K}$.
\end{theorem}
\begin{proof}  Assume first that $K$ has a  compact support. 
If $c>2$ then, due to the previous  lemma, we can apply Schauder's fixed point
theorem to $A_m:\frak{K} \to \frak{K}$ that guarantees the existence of a fixed point for $A_m$ in $\frak{K}$, which is a semi-wavefront profile for equation (\ref{17nl}). Let now $c=2$ and consider $c_j:= 2+1/j$. Since $c_j >2$,  we already know that for each $j$ there exists a semi-wavefront $\phi_j$ of equation (\ref{twe2mm}): we can 
normalise it by the condition $\phi_j(0)= 1/2 = \max_{s \leq 0} \phi_j(s)$. It is clear from (\ref{AA}) that the set 
$\{\phi_j, j \geq 0\}$ is precompact in the compact-open topology of $C_b(\R,\R)$  and therefore we can also assume that $\phi_j \to \phi_0$ uniformly on compact subsets of $\R$, where 
$\phi_0(0) =1/2= \max_{s \leq 0} \phi_0(s)$.    In addition, 
$R_j(s):= r(\phi_j(s)) \to R_0(s):= r(\phi_0(s))$ for each fixed $s \in \R$. The sequence 
$\{R_j(t)\}$ is also uniformly bounded on $\R$. All this allows us to apply Lebesgue's dominated convergence theorem in  
$$(A_{m,j}\phi_j)(t) := \frac{1}{\epsilon_j'}\left\{\int_{-\infty}^te^{z_{1,j}
(t-s)}R_j(s)ds + \int_t^{+\infty}e^{z_{2,j}
(t-s)}R_j(s)ds \right\} = \phi_j(t), 
$$
where $z_{1,j} <0< z_{2,j}$ satisfy $z^2-c_jz -b =0$. 
In this way we obtain that $A_m\phi_0 = \phi_0$ with $c=2$ and therefore $\phi_0$ is a non-negative solution 
of equation (\ref{twe2an}) satisfying condition $\phi_0(0)=1/2= \max_{s \leq 0} \phi_0(s)$.  Lemma \ref{po} shows 
that actually $\phi_0(t)>0$ for all $t\in \R$.  We claim, in addition, that $\inf_{s \leq 0}\phi_0(s) =0$ and therefore 
$\phi_0(-\infty)=0$ in view of  Lemma  \ref{bebe}. Indeed, otherwise there exists a positive $k_0$ such that 
$k_0\leq  \phi_0(t) \leq 1/2$ for all $t \leq 0$. This implies immediately 
that $k_0/4 \leq  a(t):= \phi_0(t)(1- (\phi_0*K)(t)) \leq 3/4$ for all sufficiently large negative $t$ (say, for $t \leq t_0$).  But then 
$$
\phi_0'(t) = \phi_0'(t_0) + c(\phi_0(t)-\phi_0(t_0)) +\int^{t_0}_ta(u)du \to +\infty \ \mbox{as} \ t \to -\infty, 
$$
contradicting the positivity of $\phi_0(t)$.  In consequence,  $\phi_0$ is a semi-wavefront.

Finally, in order to prove the theorem for general kernels, we can use a similar argument by constructing a sequence of compactly supported  kernels $K_j$ converging monotonically to $K$.  Indeed, set  $K_j (s) = K(s) +  \left(\int_{-\infty}^{-j} K(s) ds + \int_j^{\infty} K(s) ds\right)/(2j)$  for $s \in [-j,j]$, and set $K_j (s) = 0 $ otherwise.  
As we already proved,  for each  fixed $c \geq 2$ and $K_j$   there exists  a semi-wavefront  $\phi_j$ propagating with the velocity $c$ and satisfying the condition $\phi_j(0)= 1/2 = \max_{s \leq 0} \phi_j(s)$. 
Due to Lemma \ref{ogran}, $0< \phi_j(t) \leq U(c,K_j)$ for all $t \in \R$. 
By using the explicit form of $U(c,K_j)$ given in Lemma \ref{ogran}, it is easy to show that the sequence 
$\{\phi_j(t)\}$ is uniformly bounded on $\R$.  Thus the sequence 
$\{\phi'_j(t)\}$ is uniformly bounded on $\R$ as well, so  we can assume that 
 $\phi_j \to \phi_0 \in C_b(\R,\R)$ uniformly on compact subsets of $\R$. But then 
$\phi_0(0) =1/2= \max_{s \leq 0} \phi_0(s)$ so that, as we  have  recently seen, $\phi_0(x+ct)$ must be a semi-wavefront for equation (\ref{17nl}). 
\qed
\end{proof}

\section{Proof of the first part of Theorem \ref{Te2}} \label{FP}
In Sections \ref{FP} and \ref{S5}, we show  that the non-local KPP-Fisher equation (\ref{17nl}) can possess multiple wavefront solutions.  It is convenient to split our proof  into two stages. In the next section, we are doing all standard technical work  related to the application of the  Lyapunov-Schmidt reduction. This allows us to focus our attention in the present section on the new ideas of the proof. 

We start by analysing  zeros  of the function 
$\chi_1(z) := z- \exp(-z \tau)$: 
\begin{lem}  \label{hub} The function $\chi_1(z)$ has exactly three simple zeros (denoted as $z_1(\tau) \in (0,1), $  $z_2(\tau)$ and $z_3(\tau)= \bar z_2(\tau) \in \C $) in the half-plane $\{\Re z \geq 0\}$  and does not have any root on the imaginary axis  $\{\Re z = 0\}$  if and only if $\tau \in  (3\pi/2, 7\pi/2).$ Furthermore, $\Re z_2(\tau) < z_1(\tau)$. 
\end{lem}
\begin{proof} By applying the Rouch\'e theorem in the domains $D_R\ni \{0\}$ bounded by the graphs of $\{\Re z =-2\}$ and 
$\{|z|=R\}$, $R>0$,  we easily find that  the half-plane $\{\Re z > -2\}$ contains only one 
zero $z_1$ of $\chi_1(z)$ for every $\tau \in [0, 0.5\ln 2)$.  It is clear that $z_1>0$ if $\tau>0$.  Since $\tau >0$, all zeros of $\chi_1(z)$ are simple. This means that when $\tau$ is increasing from the initial value $0.5\ln 2$, each new pair of roots appearing in the half-plane $\{\Re z >0\}$ should cross the imaginary axis $\{\Re z =0\}$ at some moment $\tau_n$.  It is easy to check that 
the first pair of complex conjugated roots $z_2(\tau), z_3(\tau)$ will cross transversally $\{\Re z =0\}$ at the point $\tau = 3\pi/2$ with the velocity $\Re z_j'(\tau)|_{\tau = 3\pi/2} >0$.  The same happens with each other pair of roots 
crossing $\{\Re z =0\}$ at the moments $\tau_n = 3\pi/2 + 2\pi n$.  Finally, $\Re z_2(\tau) < z_1(\tau)$ for all $\tau$ such that $\tau- 3\pi/2$ is small and positive.  If $\Re z_2(\tau_0) =z_1(\tau_0)$ then $|z_2|=|\exp(-z_2\tau)|=|\exp(-z_1\tau)|=z_1$ so that $\Im z_2 =0$, a contradiction. 
\qed
\end{proof}
\begin{theorem} \label{Te2r} For each $\tau \in  (3\pi/2, 7\pi/2)$ there is $c_*(\tau)>2$ and an open subset $\Omega$ of $\R^3$ such that, for each fixed $c > c_*(\tau)$, 
the KPP-Fisher equation with advanced argument 
\begin{equation}\label{adv}
u_t(t,x) = u_{xx}(t,x)  + u(t,x)(1-u(t+\tau,x)), \ u \geq 0,\ x \in
\R,  \end{equation}  
has a three-dimensional family $u(t,x) = \phi( x+ct, \zeta,c), \ {\bf \zeta} \in \Omega,$ of wavefronts.   For each fixed $c>c_*(\tau)$,  $\phi$ maps $\Omega$ continuously  and injectively  into $C_b(\R,\R)$ and  contains a unique (up to a translation) monotone wavefront. 
\end{theorem}
\begin{proof}  By the definition,  every wavefront profile $\phi$ to equation (\ref{adv}) is a solution 
of the nonlinear boundary value problem 
\begin{equation}\label{KPPadv}
\phi''(t) - c\phi'(t) + \phi(t)(1-\phi(t+c\tau)) =0,  \quad \phi(-\infty) =0,\ \phi(+\infty)=1, \  \phi(t) >0.
\end{equation}
By setting $\epsilon = c^{-2} >0$ and realizing the  change of variables
$
y(t) = 1- \phi(-ct), 
$
we transform (\ref{KPPadv}) into the following  equivalent form:
\begin{equation}\label{pfe}
\epsilon y''(t) +y'(t) -y(t-\tau)(1-y(t)) =0,   \ y(-\infty) =0,\ y(+\infty)=1, \ y(t) <1.
\end{equation}
Taking $\epsilon =0$ in (\ref{pfe}), we obtain the first order system
\begin{equation}\label{pfel}
y'(t) = y(t-\tau)(1-y(t)),   \ y(-\infty) =0,\ y(+\infty)=1, \ y(t) <1, \ t \in \R.
\end{equation}
It is easy to see that the condition $y(t) <1$ in (\ref{pfel}) is redundant. Indeed,  if $y(t_0) =1$ at some 
leftmost point $t_0,$ then the function $z(t) =1-y(t)$ solves the linear non-autonomous equation 
$
z'(t) = -\hat a(t) z(t),\ z(t_0) =0, 
$
where $\hat a(t):= y(t-\tau)$ is bounded and continuous on $\R$. But then $z(t) \equiv 0$ and, in consequence, 
$y(t)\equiv 1$,  a contradiction. 

Furthermore,  for each nontrivial initial function $a \in C([-\tau, 0], [0,1])$, the Cauchy problem 
$y'(t) = y(t-\tau)(1-y(t)),   \ y(s)= a(s), \ s \in [-\tau,0]$, has a unique monotone solution  converging to 
$1$ as $t \to +\infty$. 
In consequence,  applying \cite[Theorem 5]{FT}, we obtain that equation (\ref{pfel}) has a positive increasing heteroclinic solution $y(t) = \phi_0(t)$.  Then Theorem \ref{theorem3.2} of Section \ref{S5} assures the following:   

For each fixed $\tau \in (3\pi/2, 7\pi/2)$ and $\frak{m} =(\mu_1,\mu_2)$ with $-1 <\mu_1 < 0 < \mu_2< \Re z_2(\tau) <1$,  there exists a small $\epsilon_0 >0$  and an open subset $\Omega$ of $\R^3$ such that, for each fixed $\epsilon \in [0,\epsilon_0],$  equation (\ref{pfe})  has a continuous  three-dimensional family of  heteroclinic solutions $\mathcal{F}(\mu_2):= \{y(t, \zeta, \epsilon), \ \zeta \in \Omega\},$ satisfying $y(t, \zeta_1, \epsilon) \not= y(t, \zeta_2, \epsilon)$ for $\zeta_1 \not= \zeta_2$,  $y(t, \mathbf{0}, 0) = \phi_0(t),$ $\sup_{s \leq
0}e^{-\mu_2 s} |y(s)| < \infty$ 
(for a moment, we do not claim that  $y(t, \zeta, \epsilon) <1$).  Moreover, 
$\mathcal{F}(\mu_2)$ contains all heteroclinic solutions of (\ref{pfe})  satisfying $|y-\phi_0|_\frak{m} < \sigma$ whenever $\sigma>0$ is sufficiently small.

This means that   for each $\tau \in  (3\pi/2, 7\pi/2)$ there is a positive $c_*(\tau)$  and an open subset $\Omega$ of $\R^3$ such that equation (\ref{KPPadv})
has a three-dimensional family $\phi(t, \zeta, c), \ \zeta \in \Omega,$ of different heteroclinic connections  for each $c > c_*(\tau)$.  Let us prove that all these connections are positive.  Indeed,  since each solution 
$\phi(t) = \phi(t, \zeta, c),$  $t \in \R,$ of (\ref{KPPadv}) 
is bounded, it should satisfy  
\begin{equation}\label{iie37} 
\phi(t) = \frac{1}{\mu-\lambda}
\int_t^{+\infty}(e^{\lambda(t-s)}- e^{\mu(t-s)})\phi(s)\phi(s+c\tau)ds,
\end{equation}
where $\lambda, \mu$ are defined in $(\ref{E})$. Next, we know that $\phi(-\infty)=0$, $\phi(+\infty)=1$,  
and therefore there exists the rightmost point $t_0\in \R \cup \{-\infty\}$ such that $\phi(t_0) =0$ and $\phi(t) > 0$ for all $t > t_0$.  But then, assuming that $t_0$ is finite  and taking $t=t_0$ in (\ref{iie37}), we get a contradiction: $0= \phi(t_0) >0$. 

Next, we claim that  the set $\{\phi(t, \zeta, c), \ \zeta \in \Omega\}$ contains a unique (up to a translation) monotone wavefront for each fixed $c > c_*(\tau)$.  In order to prove this assertion,  we take 
$0<\mu_2<\mu_2'<z_1(\tau)$ such that the strip $\Sigma(\mu_2'):=\{
z \in \mathbb{C}: \Re z \geq \mu_2'\}$ contains  exactly one zero, $z_1(\tau)$,
of $\chi_1(z)$ while the strip  $\Sigma(\mu_2)$ contains exactly three zeros, $z_1(\tau)$ and $z_2(\tau)= \bar z_3(\tau)$, of $\chi_1(z)$. 
 It is easy to see that, in such a case, $\Sigma(\mu_2')$ contains also 
exactly one root $z_1(\tau,\epsilon),$ $z_1(\tau,0):=  z_1(\tau), $ of the characteristic equation 
$\vare z^2+ z - e^{-\tau z} =0$, for all sufficiently small $\epsilon \geq 0$. Respectively, $\Sigma(\mu_2)$ contains
exactly three roots $z_j(\tau,\epsilon),$ $z_j(\tau,0):=  z_j(\tau), \ j =1,2,3 $ of the characteristic equation 
$\vare z^2+ z - e^{-\tau z} =0$. 
In addition, $z_j(\tau,\epsilon), \ j =1,2,3$ are simple and depend
continuously on $\tau,\epsilon$.  Also, with $\mu_2'$ as above, due to    Theorem \ref{theorem3.2} and Corollaries \ref{Cor3.1}, \ref{Cor3.2}   in Section \ref{S5},   the sub-family $\mathcal F(\mu_2')$ of functions 
$y(t, \zeta, \epsilon),  \ \epsilon \in [0,\epsilon_0],\  \zeta \in \Omega,$  
 such that $\sup\left\{|y-\phi_0|_\frak{m}: y \in \mathcal F(\mu_2')\right\} < \infty$ (hence, $\sup_{s \leq
0}e^{-\mu_2' s} |y(s,  \zeta, \epsilon)|$, $\sup_{s \leq
0}e^{-\mu_2' s} |y'(s,  \zeta, \epsilon)|$  are uniformly bounded)
 is 1-dimensional. This implies that each $y(\cdot, \zeta, \epsilon)\in \mathcal{F}(\mu_2')$ satisfies 
\begin{equation}\label{aft}
(y,y')(t, \zeta, \epsilon) = (1,z_1(\tau,\epsilon))C(\zeta, \epsilon)e^{z_1(\tau,\epsilon) t} + O(e^{(z_1(\tau,\epsilon)+\delta) t}), \ t \to -\infty,  
\end{equation}
for some $\delta >0$ and $C(\zeta, \epsilon)$, see e.g. \cite[Propositions 6.1, 6.2]{FA}.  Let us prove that  $C(\zeta, \epsilon)\not=0$. Indeed, if $ C(\zeta, \epsilon)=0$ then $y(\cdot, \zeta, \epsilon)\in \mathcal{F}(\mu_2')$ is a small solution in the sense that $y(t, \zeta, \epsilon) = O(e^{Lt}),\ t \to -\infty$ for each $L >0$, cf.  \cite[Proposition 6.2]{FA}. On the other hand, it is easy to see that equation  (\ref{pfe}) does not have any
nontrivial small solution. Indeed, if such a solution  $y_*(t)\not\equiv 0$ exists, the function $z_*(t) = e^{-Lt}y_*(t)$ is exponentially decreasing when $t \to -\infty$, for each fixed $L>0$. Next,  $z_*(t)$ satisfies  the asymptotically autonomous {\it linear} equation 
\begin{equation} \label{upaem}
\epsilon z''(t) +(1+2\epsilon L)z'(t) +(\epsilon L^2+L)z - z(t-\tau)e^{-L \tau}(1-y_*(t))=0, 
\end{equation}
whose limit equation at $-\infty$,
\begin{equation} \label{aem}
\epsilon z''(t) +(1+2\epsilon L)z'(t) +(\epsilon L^2+L)z - z(t-\tau)e^{-L \tau}=0, 
\end{equation}
has the characteristic equation $\epsilon (L+z)^2+ (L+z) - e^{-(L+z)\tau}=0$. Thus, for 
all $L>0$ sufficiently large, equation (\ref{aem}) is exponentially stable.  Due to the roughness property 
of an exponential dichotomy (in particular, of an exponential stability, see \cite[Lemma 4.3]{HL}), the unperturbed equation (\ref{upaem})
is exponentially stable too. This means that $z_*(t)\equiv 0$, contradicting our initial assumption of non-triviality of 
$y_*(t)$. 

Hence,  $C(\zeta, \epsilon)\not=0$ in (\ref{aft}) and therefore   $y(s,  \zeta, \epsilon), y'(s,  \zeta, \epsilon)$ do not change 
their signs at $-\infty$.  Consequently, the associated positive solutions $\phi(t, \zeta, c)$ of (\ref{KPPadv}) are eventually monotone 
at $+\infty$ and each $\phi(t)= \phi(t, \zeta, c) \not=1$ for all sufficiently large $t$.  Then either $\phi(t) >1$ on some maximal interval 
$(T, +\infty),$  $T \in \R,$ or $\phi(t) <1$ on some maximal interval $(S, +\infty), \ S \in \R\cup \{-\infty\}$.

In the first case, there exists some $t_1 \in (T, +\infty)$ such that $\phi'(t_1)=0,$ $\phi''(t_1)\leq 0,$ $\phi(t_1) >1,$ $ \phi(t_1+c\tau)>1$. But then 
 $0 \geq \phi''(t_1)= - \phi(t_1)(1-\phi(t_1+c\tau)) >0$, a contradiction. 

In the second case, suppose that $\phi'(t_2) =0$ at some rightmost point $t_2$. Then $\phi''(t_2)\geq 0,$ $\phi(t_2) <1,$ $\phi(t_2+c\tau) <1$, and we again obtain a contradiction: 
 $0 \leq \phi''(t_1)= - \phi(t_1)(1-\phi(t_1+c\tau)) <0$.  The above arguments show that if $y \in \mathcal{F}(\mu_2')$ then 
 $\phi'(t) = \phi'(t, \zeta, c) >0$ for all $t\in \R$. 
 
 Finally,  take some $y \in \mathcal{F}(\mu_2)\setminus \mathcal{F}(\mu_2')$. Then we have that $\sup_{s \leq
0}e^{-\mu_2' s} |y(s,  \zeta, \epsilon)|=\infty$, $\sup_{s \leq
0}e^{-\mu_2 s} |y(s,  \zeta, \epsilon)|<\infty$ and therefore, for some $D(\zeta, \epsilon)\not=0$, $\delta >0$, it holds that 
$$
y(t, \zeta, \epsilon) = D(\zeta, \epsilon)e^{\Re z_2(\tau,\epsilon) t}\cos(\Im z_2(\tau,\epsilon) t + E(\zeta, \epsilon)) + O(e^{(\Re z_2(\tau,\epsilon)+\delta) t}), \ t \to -\infty.
$$
This implies that all solutions  $y \in \mathcal{F}(\mu_2)\setminus \mathcal{F}(\mu_2')$ are oscillating around zero at $-\infty$ so that every monotone solution in $\mathcal{F}(\mu_2)$ belongs to 1-dimensional subfamily $\mathcal{F}(\mu_2')$.  Since small translations of each heteroclinic $y \in \mathcal{F}(\mu_2')$ leave it within $\mathcal{F}(\mu_2')$, we may conclude that the  1-dimensional subfamily $\mathcal{F}(\mu_2')$ is generated by translations of some fixed heteroclinic solution. For each fixed  sufficiently large $c$, this proves the uniqueness (up to a translation) of a monotone front in the family $\phi(t, \zeta, c)$. \qed
\end{proof}

\section{Proof of the existence of heteroclinic solutions for equation (\ref{pfe})} \label{S5}

In this section, we apply the Hale-Lin functional-analytic approach  \cite{fhw,FT,GTLMS,HL} 
 to equations (\ref{pfe}) and 
(\ref{pfel}). 
The wavefronts  for (\ref{pfe}) without the restriction $y(t)<1$ will be
obtained as perturbations of  the monotone positive heteroclinic solution $\phi_0(t)$ of
(\ref{pfel}).  Hence,  it is convenient to use  the change of variables $y(t)=w(t)+\phi_0(t)$ transforming 
(\ref{pfe})  without the restriction $y(t)<1$ into 
\begin{equation}\label{e3.3}
\epsilon w''(t)+w'(t)-w(t)=-L(t,w_t) - G(\vare,t,w_t),\quad w(-\infty)=w(+\infty)=0.
\end{equation}
Here $\vare \geq 0, \ w_t(\cdot) := w(t+\cdot)\in C[-\tau,0],$  and the functionals $G,L: \R_+\times \R \times C[-\tau,0] \to \R$ are defined by   $$G(\vare,t, v(\cdot)):=\epsilon {\phi_0}''(t) + v(0)v(-\tau), \quad
L(t,v(\cdot)) := (1+\phi_0(t-\tau))v(0)+ 
(\phi_0(t)-1)v(-\tau), 
$$
The  roots of
the characteristic equation for
$\epsilon w''(t)+w'(t)-w(t)=0$ are the extended real numbers 
$$
\al(\epsilon)={{-1-\sqrt{1+4\vare}}\over {2\vare}},\q
\be(\epsilon)={{-1+\sqrt{1+4\vare}}\over {2\vare}}\ \ \mbox{for}\ \vare >0, \ \mbox{and} \ \al(0) : = -\infty, \ \be(0) :=1. 
$$
Functions $\al(\cdot), \ \be(\cdot)$ are continuous on $\R_+$  (including $0$ because  $\al(\vare)\to -\infty,$ $ \be(\vare)\to 1^-$ as $\vare \to 0^+$). 

A bounded function $w:\mathbb{R}\to\mathbb{R}$ is a solution of
(\ref{e3.3}) if and only if
\begin{equation}\label{e3.5}
Jw(t)=H(\vare,w)(t),\q t\in\mathbb{R},
\end{equation}
$$
\hspace{-11mm}\mbox{where}\q (Jw)(t)=w(t)-\int^{+\infty}_te^{t-s} L(s,w_s)ds, \q
H(0,w)(t):=  \int^{+\infty}_te^{t-s} w(s)w(s-\tau)ds,$$ 
$$\displaylines{ \hspace{-32mm} \mbox{and, for}  \ \vare >0, \q \q \q
H(\vare,w)(t)=\int^{+\infty}_t \left[
\frac{e^{\be(\vare)(t-s)}}{\sqrt{1+4\vare}} - e^{(t-s)}\right]L(s,w_s) \, ds\ +\cr
{1\over {\sqrt{1+4\vare}}}\left [\int_{-\infty}^t{{e^{\al(\vare)(t-s)}}}(L(s,w_s) + G(\vare,{s},w_s))\, ds+ \int^{+\infty}_t
e^{\be(\vare)(t-s)} G(\vare,s,w_s)
ds\right].}$$
Our purpose is to apply a contraction principle argument in order to
obtain a solution of Eq. (\ref{e3.5}), for $\vare> 0$ small and $w$
close to 0, in the space  $C_{\frak{m}}$, for suitably chosen $\frak{m} =(\mu_1,\mu_2),$ $\mu_1 <0<\mu_2<1$. We first analyse the
linear part $J_{\frak{m}}:= J|_{C_{{\frak{m}}}}:{C_{{\frak{m}}}}\to {C_{{\frak{m}}}},$
by introducing the auxiliary operators $D_{\frak{m}},T_{\frak{m}}:C_{\frak{m}} ^1\to C_{\frak{m}}$, 
defined by 
$
(T_{\frak{m}}y)(t)=y'(t)-y(t)+L(t,y_t),$ $(D_{\frak{m}}y)(t)= y'(t) -y(t). 
$
\begin{lem}\label{L3.1}  The linear operators  $D_{\frak{m}}, T_{\frak{m}}$ and $J_{\frak{m}}$ are bounded. Moreover, $D_{\frak{m}}$ is a bijection and $T_{\frak{m}}= D_{\frak{m}}\circ J_{\frak{m}}|_{C_{\frak{m}} ^1}$.  \end{lem}
\begin{proof} By a direct computation we find that $|L(\cdot,y_{\cdot})|_{\frak{m}} \leq (2+e^{-\mu_1\tau})|y|_{\frak{m}}$, 
$$
|D_{\frak{m}}y|_{\frak{m}} \leq |y|_{1,\frak{m}}, \quad 
|T_{\frak{m}}y|_{\frak{m}} \leq (3+e^{-\mu_1\tau}) |y|_{1,\frak{m}}, \quad |J_{\frak{m}}y|_{\frak{m}} \leq \left(3+\frac{2+e^{-\mu_1\tau}}{1-\mu_2}+ e^{-\mu_1\tau}\right)|y|_{\frak{m}}. 
$$
If $y \in C_{\frak{m}} ^1$ then $(J_{\frak{m}}y)'(t) = y'(t) + L(t,y_t) + (J_{\frak{m}}y)(t) -y(t)$, so that $J_{\frak{m}}y \in C_{\frak{m}} ^1$
and $T_{\frak{m}}y= D_{\frak{m}}\circ J_{\frak{m}}y$. 
Furthermore, it can be easily seen that there exists the inverse of $D_{\frak{m}}$: 
$$
D_{\frak{m}}^{-1}y(t)= -\int_t^{+\infty}e^{t-s}y(s)ds,  \quad |D_{\frak{m}}^{-1}y(t)|_{1,\frak{m}} \leq \left(3+ \frac{2}{1-\mu_2}\right)|y|_{\frak{m}}. \hspace{2cm} \square
$$
\end{proof}
Next, consider the linear differential  equation  
\begin{equation}\label{e3.6}
y'(t)=y(t)-L(t,y_t).
\end{equation}
This equation  is
asymptotically autonomous, with the limiting equations 
$
y'(t)=  y(t-\tau)$ and  
$ y'(t) = - y(t),  
$
respectively, at $-\infty$ and $+\infty$.   
\begin{lem}\label{L3.2}  Assume that $\tau \in  (3\pi/2, 7\pi/2)$. Let  $\frak{m} = (\mu_1,\mu_2)$  satisfy
$$
-1 <\mu_1 < 0 < \mu_2< z_1(\tau) <1, \quad \mu_2  \not=\Re\, z_2(\tau). 
$$  
Then\
$
Im\, (T_{\frak{m}})=C_{\frak{m}},\ \dim\, Ker\, (T_{\frak{m}})=r_{\frak{m}},
$
where $r_{\frak{m}}=\# \{ z\in \mathbb{C} :\chi_1(z) =0, \Re\,
z
>\mu_2\}.$ 
\end{lem}
\begin{proof}
Following Hale and Lin \cite{HL}, we say that the first 
order linear autonomous delayed equation $y'(t)=M(y_t)$ has a  `shifted exponential dichotomy' on $\mathbb{R}$ with the splitting made  at $\nu \in \R$,  if the vertical line $\{\Re z = \nu\}$ does not contain any eigenvalue of  $y'(t)=M(y_t)$. Hence, clearly,  the equations $y'(t)=-y(t)$ and  $y'(t)=  y(t-\tau)$
admit  shifted exponential dichotomies on $\mathbb{R}$ with the
splitting made   at $\mu_1$ and $\mu_2$, respectively.  As a consequence, by 
\cite[Lemma 4.3]{HL}, there is $T>0$ such that  (\ref{e3.6}) has a
shifted exponential dichotomy on $(-\infty, -T]$ and $[T,\infty)$.
Therefore we can apply Lemma 4.6 of \cite{HL} to (\ref{e3.6}).  It follows that
$T_{\frak{m}}$ is a Fredholm operator, with index $Ind(T_{\frak{m}})$ given by
$$
Ind\,(T_{\frak{m}})=\dim Im\,(P_u^-(-t))-\dim Im\, (P_u^+(t)),\q t\ge T,
$$
where $P_u^-(-t), P_s^-(-t)$ and $P_u^+(t), P_s^+(t)\, (t\ge T)$ are
the projections associated with the shifted exponential
dichotomies for $y'(t)=y(t-\tau)$ and $y'(t)=-y(t)$,
respectively. From \cite[Lemma 4.3]{HL}, we also have that
$P_u^-(-t)\to P_u^-, P_u^+(t)\to P^+_u$ as $t\to\infty$, where
$P_u^-$ is the canonical projection from ${ C[-\tau,0]}$ onto the
$\mu_2$-unstable space $E_{\mu_2}^-$ for   $y'(t)=y(t-\tau)$, and $P_u^+$
is the canonical projection from ${ C[-\tau,0]}$ onto the unstable space
$E_{\mu_1}^+$ for for $y'(t)=-y(t)$. We have $E_{\mu_1}^+=\{ 0\}$ and $\dim
E_{\mu_2}^-=r_{\frak{m}}$, consequently  $Ind\, (T_{\frak{m}})=r_{\frak{m}}$.  On the other hand, the index
of $T_{\frak{m}}$ is defined by $Ind\, (T_{\frak{m}})=\dim Ker\,(T_{\frak{m}})-{\rm codim}\,
Im\,(T_{\frak{m}})$. Again by \cite[Lemma 4.6]{HL} we find that  $\dim
Ker\,(T_{\frak{m}})=\dim E_{\mu_2}^-=r_{\frak{m}}$, and therefore $Im\, (T_{\frak{m}})=C_{\frak{m}}$.
\hfill $\square$
\end{proof}
Observe that  $r_{\frak{m}}=1$ for $\mu_2$ close to $z_1(\tau)$ and  $r_{\frak{m}}=3$ for $\mu_2 < \Re z_2(\tau)$.  Moreover, 
since $T_{\frak{m}}= D_{\frak{m}}\circ J_{\frak{m}}|_{C_{\frak{m}} ^1}$ is a
surjection, we have
\begin{lem}\label{L3.3} Let $\frak{m} = (\mu_1, \mu_2)$ be as in  Lemma \ref{L3.2}.
Then  the operator  $J_{\frak{m}}:{C_{{\frak{m}}}}\to {C_{{\frak{m}}}}$ is surjective and   $Ker\, (J_{\frak{m}})=Ker\, (T_{\frak{m}})$.
\end{lem}
\begin{proof} 
 Clearly, for $w\in C_{{\frak{m}}}$ we have $J_{{\frak{m}}}w=0$ if and only if $w$ satisfies (\ref{e3.6})
 and   therefore
 $w'\in C_{{\frak{m}}}$ and  $Ker\, (J_{\frak{m}})=Ker\, (T_{\frak{m}})$. 

Next, if $y\in C_{{\frak{m}}}$ then  $\xi:= y- J_{{\frak{m}}}y \in  C_{{\frak{m}}}^1$. Equation $J_{{\frak{m}}}w =
y$ is equivalent to $ J_{{\frak{m}}}(w-y) =  \xi $ (hence, it is equivalent to $ T_{{\frak{m}}}(w-y) = D_{{\frak{m}}}\circ J_{{\frak{m}}}(w-y) =  D_{{\frak{m}}}\xi $) and therefore it possesses a
solution $\chi \in C^1_{{\frak{m}}}$.  
Thus  $J_{{\frak{m}}}(\chi+y) =y$ so that   $J_{\frak{m}}(C_{\frak{m}})= C_{\frak{m}}$.\qed 
\end{proof}
For the next stage of our analysis,  we need the detailed description of the main properties of the nonlinear operator $H$ in  (\ref{e3.5}).  
\begin{lem}\label{L3.5} Let $\frak{m} = (\mu_1, \mu_2)$ be as in Lemma \ref{L3.2} and $B_\sigma^\frak{m}(0)$ denote the $\sigma$-neighborhood of 0 in $C_{\frak{m}}$.
Then there exist  $\vare^*>0$ and non-negative continuous functions $C(\vare, \sigma), D(\vare),$ $\sigma \geq 0,$ $ \vare \in [0, \vare^*),$    such that  $C(0,0) = D(0) =0$, and for any $\vare \in [0, \vare^*)$ and  $w\in B_\sigma^\frak{m}(0)$,    it holds 
\begin{equation}\label{e3.7}  
|H(\vare,w)|_{\frak{m}} \le  C(\vare, \sigma)|w|_{\frak{m}} + D(\vare), \
 |H(\vare,w)-H(\vare,v)|_{\frak{m}} \le C(\vare, \sigma) |w-v|_{\frak{m}}.
 \end{equation}
 Furthermore, 
 $H: [0, \epsilon^*)\times  B_\sigma^\frak{m}(0) \to  C_{\frak{m}}$ is a continuous function. 
\end{lem}
\begin{proof}
 We write $H=H_1+H_2+H_3$, where $H_1(0,w)= H_3(0,w)\equiv 0$, $H_2(0,w) =H(0,w)$ and, for $\vare >0$, 
 $$
H_1(\vare,w)(t) =\int^{+\infty}_t \left[
\frac{e^{\be(\vare)(t-s)}}{\sqrt{1+4\vare}}  - e^{(t-s)}\right]L(s,w_s)ds;$$
$$H_2(\vare,w)(t)  ={1\over {\sqrt{1+4\vare}}}\int^{+\infty}_t{{e^{\be(\vare)(t-s)}}}G(\vare,s,w_s)ds;  $$
$$H_3(\vare,w)(t) ={1\over {\sqrt{1+4\vare}}}\int_{-\infty}^t{{e^{\al(\vare)(t-s)}}}(L(s,w_s) + G(\vare,s,w_s))ds.$$ 
For
$t\in\mathbb{R},\ \vare> 0$, $j=1,2,$ we have
$$
\left |\int^{+\infty}_t \left(
\frac{e^{\be(\vare)(t-s)}}{\sqrt{1+4\vare}} - e^{(t-s)}\right)e^{\mu_js}ds\right|=
  \left [{{1}\over {\sqrt{1+4\vare}}}
 {{1}\over {\be(\vare)-\mu_j}}-{ {1}\over {1-\mu_j}}\right] e^{\mu_j
 t}=:  c_j(\vare)e^{\mu_j t},$$ where $ c_j(0^+)=0. 
 $
 As a consequence, setting $c_3(\vare):= (c_1(\vare)+c_2(\vare))
(2+e^{-\mu_1\tau})$, we obtain
\begin{equation}\label{e3.9}
 H_1(\vare,w)\in C_\frak{m}, \ |H_1(\vare,w)-H_1(\vare, v)|_\frak{m}\le c_3(\vare)|w-v|_\frak{m},\q w,v\in C_\frak{m},\vare> 0.
\end{equation}
Next, for
$t\in\mathbb{R},$ $\vare \geq 0,$ $ w,v\in  B_\sigma^\frak{m}(0)$, 
we have 
$$|G(\vare,t,w_t)|=|\epsilon {\phi_0}''(t) + w(t)w(t-\tau)|\le
\vare|\phi_0''(t)|+ \sigma |w(t)|, $$ 
$$
|G(\vare, t,w_t)-G(\vare,t,v_t)|\le
\sigma (|w(t)-v(t)| + |w(t-\tau)-v(t-\tau)|).
$$
Now, since the equilibria $0, 1$ of equation (\ref{pfel}) are hyperbolic (cf. Lemma \ref{hub}),  $\phi_0(t)$  converges to the limits 
$\phi_0(+\infty)=1$ and $\phi_0(-\infty)=0$ at exponential rate.  In fact, 
there exist finite $\lim_{t\to +\infty} (1 - \phi_0(t))e^{t}$ and 
$\lim_{t\to -\infty}\phi_0(t)e^{-z_1(\tau)t}$, see e.g. \cite{GT} for more details. As a consequence, we conclude from  
$\phi_0'(t) = \phi_0(t-\tau)(1-\phi_0(t))$, $\phi_0''(t) = \phi_0'(t-\tau)(1-\phi_0(t)) - \phi_0(t-\tau)\phi_0'(t)$ that  $\phi_0', \phi_0''\in C_\frak{m}$. It follows from the above estimates that, for all $v, w\in   B_\sigma^\frak{m}(0)$, $\vare \geq 0$,  
\begin{eqnarray}\label{e3.13}
\nonumber & & |H_2(\vare,w)|_{\frak{m} } \le {2\over
{(\be(\vare)-\mu_2)\sqrt{1+4\vare}}}(\vare|{\phi_0}''|_{\frak{m} }+\sigma
|w|_{\frak{m} }), \\
\nonumber   & & |H_2(\vare,w)-H_2(\vare, v)|_{\frak{m} }\le{2\sigma(1+ e^{-\mu_1\tau}) \over
{(\be(\vare)-\mu_2)\sqrt{1+4\vare}}} |w-v|_{\frak{m} }, \\
\nonumber  & & |H_3(\vare,w)|_{\frak{m} } \le {2\over
{(\mu_1-\al(\vare))\sqrt{1+4\vare}}} \Big
[\vare|{\phi_0}''|_{\frak{m} }+(2+e^{-\mu_1 \tau}+\sigma)|w|_{\frak{m} }\Big],\\
\nonumber   & & |H_3(\vare,w)-H_3(\vare, v)|_{\frak{m} }\le {2{(2+e^{-\mu_1 \tau})(1+\sigma)}\over
{(\mu_1-\al(\vare))\sqrt{1+4\vare}}}|w-v|_{\frak{m} }.
\end{eqnarray}
 From these inequalities, for $\vare\geq 0$ small enough we
obtain  that (\ref{e3.7}) holds for all 
$ w,v\in B_\sigma^\frak{m}(0),
$
with $C(\vare, \sigma), D(\vare)$ given by
$$C(\vare,\sigma)=c_3(\vare)+{2\sigma(1+ e^{-\mu_1\tau}) \over
{(\be(\vare)-\mu_2)\sqrt{1+4\vare}}}+{{2(2+e^{-\tau\mu_1})(1+\sigma)}\over
{(\mu_1-\al(\vare))\sqrt{1+4\vare}}},\q 
$$
$$
D(\vare)=\Big({1\over
{\be(\vare)-\mu_2}}+ {1\over
{\mu_1-\al(\vare)}} \Big)
\frac{2\vare|{\phi_0}''|_{(-1,z_1(\tau))}}{\sqrt{1+4\vare}} .$$ 

\noindent Since  $c_3(0)= 0,
\al(0^+)=-\infty$, we obtain that $C(0,0) = D(0) =0$.

Finally, it remains to prove that  the function 
 $H: [0, \epsilon^*)\times  B_\sigma^\frak{m}(0) \to  C_{\frak{m}}$ is continuous. It is easy to show that $H(\epsilon, w) \to H(\epsilon_0, w)$ in $C_{\frak{m}}$ as $\epsilon \to \vare_0$,  uniformly with respect to $w$ from bounded subsets of $C_{\frak{m}}$.  For instance, 
 the proof of such a convergence $H_1(\epsilon, w) \to H_1(0, w), \ \vare \to 0^+,$ follows from (\ref{e3.9}).  But then, due to 
 $(\ref{e3.7})$, the mapping  $(\epsilon, w) \to H(\epsilon, w)$ is continuous in $\epsilon$, $w$. 
\qed
\end{proof}
Next, for $\vare\geq 0$ small, we look for a solution $w\in C_{\frak{m}}$ of (\ref{e3.5}). 
We first apply a Lyapunov-Schmidt reduction. From Lemmas \ref{L3.2}
and \ref{L3.3}, it follows that $X_{\frak{m}}:=Ker\, (J|_{C_{\frak{m}}})$ is finite
dimensional, hence there is a complementary subspace $Y_{\frak{m}}$ in
$C_{\frak{m}}$ such that 
$
C_{\frak{m}}=X_{\frak{m}} \oplus Y_{\frak{m}}.
$
For $w\in C_{\frak{m}}$, write $w=\xi+\eta$ with $\xi\in X_{\frak{m}},\ \eta\in
Y_{\frak{m}}$. Define $S_{\frak{m}}:=J_{\frak{m}}|_{Y_{\frak{m}}}$. Since $S_{\frak{m}}:Y_{\frak{m}}\to C_{\frak{m}}$
is bounded and bijective,  $S_{\frak{m}}^{-1}$ is bounded. In the space
$C_{\frak{m}}$,  (\ref{e3.5}) is equivalent to $\eta
=S_{\frak{m}}^{-1}H(\vare, \xi+\eta)$, therefore we look for fixed points
$\eta\in Y_{\frak{m}}$ of the map
\begin{equation}\label{e3.17}
{\cal F}_{\frak{m}}(\vare,\xi,\eta)=S_{\frak{m}}^{-1}H(\vare, \xi+\eta).
\end{equation}
The following result is straightforward.
 \begin{theorem} \label{theorem3.2} Let $\frak{m} = (\mu_1,\mu_2)$  and $-1 < \mu_1 < 0 < \mu_2 < z_1(\tau)$ be such that there are no zeros of $\chi_1(z)$  with $\Re z =\mu_2$.  Then there exist $\vare^*>0$,  $\sigma >0$, such that  the following  holds:  for each fixed $\vare \in [0, \vare^*]$, the set of all 
wavefronts $\psi$ to  (\ref{pfe})  satisfying $|\psi-\phi_0|_\frak{m} < \sigma$ forms a $r_\frak{m}$-dimensional
manifold
 $${\cal M}_{\frak{m}, \vare}=\{ \psi: \psi=\phi_0+\xi +\eta (\vare, \xi),\ {\rm for}\ \xi\in X_\frak{m}\cap B_\sigma ^\frak{m}(0)\} ,$$
where  $\eta (\vare, \xi)$ is the fixed point
of ${\cal F}_\frak{m}(\vare,\xi,\cdot)$ in $Y_\frak{m}\cap B_\sigma^\frak{m} (0)$ such that $\eta (0, 0)=0$ 
and the function $(\vare, \xi)\in 
[0,\vare^*]\times (X_\frak{m}\cap \overline{B_\sigma^{\frak{m}} (0)}) \to \eta (\vare, \xi) \in C_\frak{m}$ is continuous.
\end{theorem}
\begin{proof} 
 Fix  $k\in (0,1)$. From Lemma
\ref{L3.5}, there are $\sigma>0$ and
$\vare^*>0$ such that for $0\leq \vare\leq \vare^*$, $\xi \in
X_\frak{m}\cap \overline{B_\sigma^\frak{m} (0)}$ and $\eta_1,\eta_2\in Y_\frak{m}\cap
\overline{B_\sigma ^\frak{m}(0)}$ we have
$$
 |S^{-1}_\frak{m}H(\vare, \xi+\eta_1)|_{\frak{m}}\le   \|S^{-1}_\frak{m}\| \left(C(\vare,\sigma)|\xi+\eta_1|_\frak{m}+D(\vare)\right) < \sigma, \q {\cal F}_{\frak{m}}(0,0,0)=0, 
$$
$$
 |S^{-1}_{\frak{m}}(H(\vare, \xi+\eta_1)-H(\vare, \xi+\eta_2))|_\frak{m} \le  C(\vare,\sigma)\|S^{-1}_{\frak{m}}\| |\eta_1-\eta_2|_\frak{m} \leq k|\eta_1-\eta_2|_\frak{m}. 
$$

\vspace{3mm}

\noindent Hence, 
${\cal F}_{\frak{m}}:[0,\vare^*]\times (X_\frak{m}\cap \overline{B_\sigma^{\frak{m}}
(0)})\times (Y_\frak{m}\cap \overline{B_\sigma ^{\frak{m}}(0)})\to Y_\frak{m}\cap
\overline{B_\sigma^{\frak{m}} (0)}$  is a uniform contraction map of $\eta \in
Y_\frak{m}\cap \overline{B_\sigma^{\frak{m}} (0)}$. Therefore  for $(\vare, \xi)\in
[0,\vare^*]\times (X_\frak{m}\cap \overline{B_\sigma^{\frak{m}} (0)})$ there is a
unique solution $\eta (\vare, \xi)\in Y_\frak{m}$
of (\ref{e3.17}), which depends continuously on $\vare, \xi$.   \qed
\end{proof}
\begin{cor}\label{Cor3.1} If  $0<\mu_2<z_1(\tau)$ is such that the strip $\{
z \in \mathbb{C}: \Re z \in [\mu_2, z_1(\tau))\}$ does not contain zeros
of $\chi_1(z)$, then the manifold ${\cal M}_1={\cal M}_{\frak{m}, \vare}$ is 1-dimensional.
If $\mu_2>0$ is small and $\tau \in (3\pi/2, 7\pi/2)$, then the manifold ${\cal M}_3={\cal M}_{\frak{m}, \vare}$ is 3-dimensional. Moreover, ${\cal M}_1 \subset {\cal M}_3$. \end{cor}
\begin{cor}\label{Cor3.2} Under the assumptions of Theorem
\ref{theorem3.2} (and with the same notation) there
is 
$C>0$ such that  the function $\eta (\vare,\xi)$
satisfies
\begin{equation}\label{e51}
|\eta (\vare,\xi)|_\frak{m} \le C, \q |\eta '(\vare,\xi)|_\frak{m} \le C\q
{\rm for}\q 0\leq \vare\leq \vare^*,\  \xi \in X_\frak{m}\cap \overline{B_\sigma
^{\frak{m}}(0)}.
\end{equation}
\end{cor}
\begin{proof}   Since the function $\eta(\vare, \xi)$ is continuous on the compact set $ 
[0,\vare^*]\times (X_\frak{m}\cap \overline{B_\sigma^{\frak{m}} (0)})$,  the first estimate in (\ref{e51}) with  $C$ independent  of $\vare,\xi$ is obvious. 

Next, as we know,  $\phi_0', \phi_0''\in C_\frak{m}$. Similarly, {$\xi', \xi''\in C_\frak{m}$} because 
$$\xi'(t)=-\phi_0(t-\tau)\xi(t)-
(\phi_0(t)-1)\xi(t-\tau).$$
 In addition, 
 since $\psi (t):=\psi (\vare,\xi)(t)= \phi_0(t)+\xi(t) +\eta (\vare, \xi)(t)$ is a bounded solution of (\ref{pfe}),
we find that $\epsilon \eta'' +\eta' - \eta = (N\eta)$, where 
$(N\eta)(t):= - \epsilon(\phi_0''(t)+\xi''(t)) -\xi(t-\tau)(\eta(t)+\xi(t)) - (1+\phi_0(t-\tau)+ \eta(t-\tau))\eta(t) + \eta(t-\tau)(1-\phi_0(t) - \xi(t))$ satisfies, for some positive $C$, the inequality $|N\eta(\epsilon,\xi)|_\frak{m} \leq C$ for all $\xi \in X_{\frak m}\cap \overline{B_\sigma^{\frak{m}}(0)}$ and $0\leq \vare\leq \vare^*$. Consequently,  for $\epsilon >0$, 
$$
 \eta (t)={1\over {\sqrt{1+4\vare}}}\left ( \int_{-\infty}^t
e^{\al(\vare)(t-s)} (N\eta)(s)ds+ \int^{+\infty}_t
e^{\be(\vare)(t-s)}(N\eta)(s)ds\right ), 
$$
from which we derive
$$ \eta'(\epsilon,\xi)(t)={1\over {\sqrt{1+4\vare}}}\Big
(\al(\vare)\hspace{-2mm} \int\limits_{-\infty}^t
e^{\al(\vare)(t-s)}(N\eta)(s)ds
 -\be(\vare)\hspace{-2mm}\int\limits^{+\infty}_t e^{\be(\vare)(t-s)}(N\eta)(s)
ds\Big ).
$$
We also have that
$
\eta'(0,\xi) = \eta + N\eta(0,\xi)$. 
Thus  there is
$C_1>0$ independent of $\epsilon, \xi$ and such that $|\eta'(\epsilon,\xi)|_\frak{m} \le C_1$ for all $\xi \in X_{\frak m}\cap \overline{B_\sigma^{\frak{m}}(0)}$ and $\vare \in [0,\vare^*]$. 
This completes the
proof. \hfill $\square$
\end{proof}

\section{Proof of the second part of Theorem \ref{Te2}} \label{S6}

In this section, we prove that the non-local KPP-Fisher equation (\ref{17nl})  can possess fast semi-wavefronts connecting trivial equilibrium and positive periodic solution oscillating around $1$:
\begin{theorem} \label{Te2A} For each $\tau > 3\pi/2$ close to $3\pi/2$ there is $c_*(\tau)>2$ such that equation (\ref{adv}) has  proper semi-wavefronts $u(t,x) = \psi(x+ct, c)$.  The profiles $\psi(\cdot, c)$ are asymptotically  periodic at $+\infty$, with $\omega(c)$-periodic limit functions having periods $\omega(c)$ close to $2\pi c$ and of the sinusoidal form (i.e. oscillating around 1 and having exactly two critical points on the period interval $[0, \omega(c))$). 
\end{theorem}
\begin{remark} In fact, with some more effort, it is possible to establish the existence of 2-dimensional family of proper semi-wavefronts for the above mentioned KPP-Fisher equation, cf. \cite{HW}.  \end{remark}

Our proof  of the existence of a point-to-periodic connection  is based on the perturbation techniques developed by J. Hale in \cite{Hale}, \cite[Section 10.4]{hale} and
W. Huang {\it et al.} in \cite{DH,HW}. In fact, the paper \cite{HW} deals precisely with the problem of point-to-periodic connections for equations with time delay and nonlocal response. However, since there are important 
differences between the frameworks of \cite{HW} and the present paper, the main results from  \cite{HW} do not apply directly  to equation  (\ref{KPPadv}). Still, using the Krisztin-Walther-Wu theory of  delayed monotone positive feedback equations \cite{KWW},  it is possible 
to retrace the main arguments of \cite{Hale,HW} in order to obtain the desired point-to-periodic connections in our case. We are doing this work 
in the present section, where we are paying  special attention to  the arguments which are different from those used in \cite{HW}. The related results are given in Lemmas \ref{L17}, \ref{L18}, \ref{P1}, see also  Remarks \ref{R3}, \ref{R4} below. The final part of this section (after Lemma \ref{P1}) follows closely the arguments  of \cite{Hale,hale,HW}:  for completeness of the exposition, we included this part  as well.    

Analogously to the proof of Theorem \ref{Te2r}, a point-to-periodic connection in equation (\ref{KPPadv})
is obtained as a result of singular perturbation of a  periodic-to-point 
connection $\phi_0$ for  the equation 
\begin{equation}\label{wri}
y'(t) = y(t-\tau)(1-y(t)). 
\end{equation}
This is possible when equation (\ref{wri}) possesses an hyperbolic $\omega-$periodic solution $p(t)$ oscillating around $0$. Our first 
result below, Lemma \ref{L17}, considers this aspect of the problem. Recall that 
the $\omega-$periodic solution $p(t)$ of (\ref{wri}) is hyperbolic if and only if  the linearised $\omega-$periodic equation 
\begin{equation}\label{lp}
z'(t)= - p(t-\tau)z(t) + (1-p(t)) z(t-\tau) 
\end{equation}
has only one Floquet multiplicator $\mu =1$ on the unit circle and, in addition, the realified generalised eigenspace $G_\R(1)$ of this multiplicator is one-dimensional: $G_\R(1)= \{cp', c \in \R\}$.  The hyperbolicity of $p(t)$ implies that the formal adjoint equation \cite{hale,HL}
$$
v'(t)= p(t-\tau)v(t) - (1-p(t+\tau)) v(t+\tau) 
$$
associated with (\ref{lp}) has a unique nonzero $\omega-$periodic solution $v(t)=p_*(t)$ normalised 
by the condition $\int_0^\omega p'(t) p_*(t)dt=1$, see e.g. \cite[pp. 1236-1237]{HW}.  Another consequence of the 
hyperbolicity of $p(t)$ is that equation (\ref{lp}) has a shifted exponential dichotomy on $\R_-$ with  exponents 
$\alpha_1=0 <  \beta_1$ \cite{HL} (as Lemma \ref{L17} shows the unstable space of this dichotomy is one-dimensional). 

Following \cite[Chapter 5]{KWW} and \cite[p. 480]{mps2}, we will say that solution $z(t)$ of equation (\ref{wri}) is slowly oscillating on $[T, +\infty)$ if,  for each fixed $t\geq T$, the function $z(t+s), \ s \in [-\tau,0],$ 
has precisely 1 or 2 sign changes on the interval $[-\tau,0]$ (a continuous function $z(t)$ has a sign change at some point $t_0$ if $z(t_0+\epsilon)z(t_0-\epsilon) <0$ for all small $\epsilon >0$, in particular, $z(t_0)=0$).

\begin{lem} \label{L17}There exists $\tau_0 >3\pi/2$ such that, for every $\tau \in (3\pi/2, \tau_0)$,  equation (\ref{wri}) 
has a  nonconstant hyperbolic periodic solution $p(t) <1, \ t \in \R,$ slowly oscillating around $0$ and 
a periodic-to-point connection $\phi_0(t)<1,  \ t \in \R,$ such that, for some $a \in (0, \beta_1)\cap (0,1)$ and $C>0$, it holds  
$$
|\phi_0(t)-p(t)| \leq Ce^{2at}, \ t \leq 0, \quad \phi_0(+\infty) =1. 
$$
\end{lem}
\begin{proof} The change of variables $1-y(t) = e^{z(t)}$ transforms (\ref{wri}) and the boundary restrictions on $\phi_0$ into the following equation:
$$
z'(t)= F(z(t-\tau)),\quad  F(z): = e^{z}-1, \quad z(t) \ \mbox{is asymptotically periodic at} \ -\infty,  \ z(+\infty) = -\infty. 
$$
Since function $F:\R\to \R$ is bounded from below, $F(0)=0,\ F'(z)> 0$ for all $z\in \R$, we can say that the equation $z'(t)= F(z(t-\tau))$ possesses  delayed positive feedback. 
For $\tau \in  (3\pi/2, 7\pi/2)$, this type of equations was thoroughly analysed in the monograph \cite{KWW} where it was proved that the equation $z'(t)= F(z(t-\tau))$ (i) has a periodic solution  $q(t)$ slowly oscillating around $0$ \cite[Corollary 5.8 and Theorem 17.3]{KWW}; (ii) has a solution $Q(t)$ such that $Q(t)-q(t) \to 0$ at $-\infty$ and 
$Q(t) \to -\infty$ as $t\to +\infty$ \cite[Theorem 17.3]{KWW}. Next, the solution $q(t)$ is the unique non-trivial periodic solution belonging to the closure of the unstable manifold of the equilibrium $z(t)\equiv 0$ in the phase space $C[-\tau,0]$ \cite[Theorem 17.3]{KWW}.  The stability properties of $q(t)$ were analysed in Chapter 8 of  \cite{KWW}. It was proved that the associated Floquet map has exactly one Floquet multiplier (of multiplicity $1$) outside the unit disc $\{z: |z| \leq 1\}$ \cite[Theorem 8.2]{KWW}. Moreover, the only Floquet multiplier  on the unit circle $\{z: |z| = 1\}$ is $1$ while the realified generalised eigenspace $G_\R(1)$ of $1$ is either one-dimensional or two-dimensional  \cite[Corollary 8.4]{KWW}. In the case, when $G_\R(1)$ is two-dimensional,  the equation $z'(t)= F(z(t-\tau))$ cannot have slowly oscillating solutions exponentially converging to $q(t)$ at $+\infty$, see  \cite[Corollary 8.4 (iv)]{KWW}  and   the proof of  Theorem 8.2 in \cite{KWW} for more details. We are going to use the latter information in order to show that 
dim $G_\R(1)=1$ when $\tau_0 >3\pi/2$ is sufficiently close to $3\pi/2$.  Indeed,  for such $\tau_0$ that  $\tau_0-3\pi/2 >0$ is small, equation (\ref{wri}) was analysed in \cite[Section 3]{FM} by means of the normal form approach. In particular, it was proved that when the parameter $\tau$ increases and passes through the point $\tau_1=3\pi/2$,  equation (\ref{wri}) undergoes a super-critical generic Hopf bifurcation from the zero equilibrium, with associated periodic solution $p(t)$ being exponentially stable with asymptotic phase in the center manifold of the trivial equilibrium, see  \cite[Example 3.24]{FM}.  Moreover, it was established that $p(t)$ oscillates slowly around $0$, in fact, 
\begin{equation}\label{sinel}
p(t) = \sqrt{\frac{20(\tau-3\pi/2)}{9\pi/2+1}}\cos\left((1+O(\sqrt{\tau-3\pi/2}))t)\right) +O({\tau-3\pi/2}). 
\end{equation}
Since the change of variables $1-y(t) = e^{z(t)}$ preserves all the above mentioned stability and oscillation properties of the periodic solution $p(t)$ and the zero steady state, we may conclude that  $1-p(t) = e^{q(t+t_q)}$ for some $t_q \in \R$ and that the unstable manifold of the trivial equilibrium to $z'(t)= F(z(t-\tau))$ contains  slowly 
oscillating solutions exponentially converging to $q(t)$. As we have already mentioned this behaviour is not possible when dim $G_\R(1)=2$. 
 Thus dim $G_\R(1)=1$ for all $\tau >3\pi/2$ sufficiently close to $3\pi/2$.   This means that $q(t)$ is a hyperbolic periodic solution of equation  $z'(t)= F(z(t-\tau))$. In particular, $Q(t)-q(t) \to 0$ exponentially as $t \to-\infty$, see \cite[Appendices I and V]{KWW}. 
Now, since the linear monodromy maps  associated with the solutions $q(t)$ and $p(t)$ are conjugate via an invertible multiplication operator, we conclude that $p(t)$ is a hyperbolic periodic solution of (\ref{wri}), too. It is clear then that $\phi_0(t) = 1- e^{Q(t)}$ is a heteroclinic connection possessing all properties mentioned in the statement of the lemma (the inequalities $\phi_0(t) <1, \ p(t) <1$ were already established  in the proof of Theorem \ref{Te2r}, in the paragraph below  formula (\ref{pfel})). 
\hfill $\square$
\end{proof}
Set now
$$(Lw)(t)=(1+\phi_0(t-\tau))w(t) - (1-\phi_0(t)) w(t-\tau), \q (Jw)(t)=w(t)-\int^{+\infty}_te^{t-s} (Lw)(s)ds.$$
The next stage of the proof concerns the solvability of the linear  inhomogeneous equations $
(Jw)(t)  = g(t)$ and 
\begin{equation}\label{deQ}
 w'(t)-w(t) =  -(Lw)(t) +g(t)
\end{equation}
in the space 
$$
X^a= \left\{g \in C_b(\R,\R): g(+\infty)= 0 \ \mbox{and there exists an $\omega-$periodic}\  \right. $$ 
$$\left. \mbox{ function} \ g_\infty(t)  \ \mbox{such that}\ \lim_{t\to -\infty}|g(t)-g_\infty(t)|e^{-at} =0 \right\},
$$
equipped with the complete norm
$$
|g|_a = \sup_{t \in \R}|g(t)| + \sup_{t \leq 0}|g(t)-g_\infty(t)|e^{-at}. 
$$
Here $a \in (0, \min\{1, \beta_1\})$ is chosen as in Lemma \ref{L17} and $g\to g_\infty$ is a linear operator transforming  function $g$,  asymptotically periodic at $-\infty$, into its periodic limit $g_\infty$ (i.e. $\lim_{t \to - \infty} |g(t)-g_\infty(t)| =0$, $g_\infty(t)=g_\infty(t+\omega), \ t \in \R$). In particular, we have $p(t)=(\phi_0)_\infty(t)$. 
We also notice that $1-\phi_0, \phi'_0, \phi''_0 \in X^a$ in view of  Lemma \ref{L17} and (\ref{wri}), however, $\phi_0 \not \in X^a$.  
\begin{remark} \label{R3} It is worth noticing that the definition of the Banach space $X^a$ given in \cite{HW} uses the restriction  
$\sup_{t\leq 0}|g(t)-g_\infty(t)|e^{-at} <\infty$ instead of\  \ $\lim_{t\to -\infty}|g(t)-g_\infty(t)|e^{-at} =0$. The advantage 
of our definition of $X^a$ is that the translation operator $T: \R\times X^a \to X^a$ defined by $(T_hw)(t)=w(t+h)$   is a continuous  function of $h, w$.  Indeed, set  $\Omega(h,w) =\sup_{t\in \R}|w(t+h)-w(t)|$, then 
$$
|w(t+h)|_a \leq 3e^{a|h|}|w|_a, \quad |w(t+h)-w(t)|_a \leq \Omega_a(h,w), 
$$ 
where 
$
\Omega_a(h,w) := \Omega(h,w)+e^{ah}\Omega(h,(w(\cdot)-w_\infty(\cdot))e^{-a\cdot})+ |e^{ah}-1||w|_a$ and 
$\Omega_a(0^+,w)=0. 
$
Thus $|T_{h_1}w_1-T_{h_0}w_0|_a \leq 3e^{a(|h_1|+|h_0|)}\left\{|w_1-w_0|_a+\Omega_a(h_1-h_0,w_0)\right\}.$ 
\end{remark}
\begin{lem} \label{L18}Suppose that $g \in X^a$. Then equation (\ref{deQ}) has a solution $w\in X^a$ if and only if 
$<g_\infty, p_*>:= \int_0^\omega g_\infty(s)p_*(s)ds=0$. 
\end{lem}
\begin{proof} First, we recall that the $\omega-$periodic hyperbolic inhomogeneous equation 
\begin{equation}\label{lim}
 w'(t)= -p(t-\tau)w(t) + (1-p(t)) w(t-\tau) + g_\infty(t),
\end{equation}
has an $\omega-$periodic solution $w_g$ if and only if $<g_\infty, p_*>=0$, 
see e.g. 
 \cite[p. 1236]{HW}. 
 
Suppose now that equation (\ref{deQ}) has a solution $w\in X^a$. After taking limit at $-\infty$ in an equivalent integral form of (\ref{deQ}) with $g \in X^a$, we find that $w_\infty(t)$ is an $\omega-$periodic solution of  (\ref{lim}). Hence, $<g_\infty, p_*>=0$.

Next,  suppose  $<g_\infty, p_*>=0$. Then equation (\ref{lim}) has an $\omega-$periodic solution $w_\infty(t)$. 
Let $w_0(t)$ be a smooth function such that $w_0(t)=w_\infty(t)$ for all $t \leq 0$ and $w_0(t)=0$ for all $t \geq \omega$.  Clearly, the function $w(t) =w_0(t) +v(t)$ is a solution of  (\ref{deQ}) if and only if $v(t)$ is a solution of 
equation 
\begin{equation}\label{eqv}
v'(t)= -\phi_0(t-\tau)v(t) + (1-\phi_0(t))v(t-\tau) + g_1(t)
\end{equation}
where  
$
g_1(t) = g(t) +[-w_0'(t)  -\phi_0(t-\tau)w_0(t) + (1-\phi_0(t))w_0(t-\tau)]. $
Observe that, for all $t \geq \omega + \tau$, we have that $g_1(t)=g(t)$, while,  for all $t \leq 0$,  
$$
g_1(t)
= g(t)-g_\infty(t) -[\phi_0(t-\tau)-p(t-\tau)]w_0(t) -[\phi_0(t)-p(t)]w_0(t-\tau).
$$
In particular, $g_1(+\infty)=0$ and $\sup_{t \leq 0}|g_1(t)|e^{-at} < \infty$. Consequently, the sufficiency of the condition $<g_\infty, p_*>=0$ for the solvability  of equation  (\ref{deQ}) with $g \in X^a$ will be established if we prove that for each $g_1\in C_\frak{m}$, $\frak{m} = (0, a)$, $g_1(+\infty)=0$, equation (\ref{eqv}) has a solution $v\in 
C_\frak{m}$ such that  $v(+\infty)=0$. To this end, we will  use  results (as well as notation, see $\alpha_j, \beta_j, \gamma_j$ below) from the Hale-Lin work \cite{HL}. By the roughness Lemma 4.3 in \cite{HL}, there exist a small $\varepsilon >0$ and large $T_*>0$ such that 
the homogeneous part of equation  (\ref{eqv}) has a shifted dichotomy on $(-\infty, -T_*]$ with exponents 
$\alpha_1=\varepsilon < a:=\gamma_1 < \beta_1-\varepsilon$  and it is exponentially stable on $[T_*,+\infty)$ (more formally, it has a shifted dichotomy on $[T_*,+\infty)$ with exponents 
$\alpha_2=-1+\varepsilon   < 0=: \gamma_2 < \beta_2:=+\infty$), Lemmas 4.5 and 4.6 from \cite{HL} assure 
that equation  (\ref{eqv}) has a solution  $v\in C^1_\frak{m}$ for each $g_1\in 
C_\frak{m}$  satisfying the orthogonality condition 
$$
\int_{-\infty}^{+\infty}g_1(t)y_o(t)dt =0 \quad \mbox{for all} \ y_o \in \frak{O},
$$
where $\frak{O}$ denotes the set of the solutions $y_o(t)$  of the formal adjoint equation  to (\ref{eqv}) 
\begin{equation}\label{fae}
y'(t)= \phi_0(t-\tau)y(t) - (1-\phi_0(t+\tau))y(t+\tau) 
\end{equation}
such that $|y_o(t)| \leq Ke^{-\beta t},\ t \geq 0$, $|y_o(t)| \leq Ke^{-\varepsilon t},\ t \leq 0$, with some positive $K, \beta\leq \beta_2$. 

We claim that $\frak{O}=\{0\}$ and therefore the  above orthogonality condition is automatically satisfied. Indeed, suppose that  $y_o\in \frak{O}\setminus\{0\}$. 
Since $y_o(+\infty)=0$,  there exists an increasing sequence $t_j \to +\infty$ such that 
$|y_o(t_j)| = \max_{t \geq t_j}|y_o(t)|>0$. Then each function $z_j(t)=y_o(t+t_j)/|y_o(t_j)|, t \geq 0,$ is uniformly bounded  by $1$  on $\R_+$ and also satisfies the equation
$$
z'(t) = \phi_0(t-\tau+t_j)z(t) - (1-\phi_0(t+\tau+t_j))z(t+\tau), \ j \in \N. 
$$
 In particular, $|z_j'(t)| \leq 3\sup_{s \in \R}|\phi_0(s)|$ for $t \geq 0, \ j \in \N$, that implies that $z_j(t)$ has a subsequence $z_{j_k}(t)$ uniformly converging on compact subsets of $\R_+$ to some nontrivial bounded solution $z_*(t),\ |z_*(0)|=1,$ of the limit equation (at $+\infty$) $z'(t) = z(t)$. Obviously, since $\max_{s\geq 0}|z_*(s)|=1,$ this cannot happen and therefore $\frak{O}=\{0\}$. 
 
Hence, equation  (\ref{eqv}) has a solution  $v\in C^1_\frak{m}$ for each $g_1\in 
C_\frak{m}$.  In this way, the lemma will be proved if we show that $v(+\infty)=0$ and $v(t)e^{-at} \to 0$ as $t \to -\infty$.  
The property $v(+\infty)=0$ becomes evident if we observe that $v'(t)=-v(t) +g_2(t)$, where $g_2(t) = (1-\phi_0(t-\tau))v(t) + (1-\phi_0(t))v(t-\tau) + g_1(t)$ satisfies $g_2(+\infty)=0$. Indeed, we have 
$$
|v(t)| =\left|v(s)e^{-(t-s)}+\int_s^te^{-(t-u)}g_2(u)du\right| \leq |v(s)|e^{-(t-s)} + \sup_{u\geq s}|g_2(u)|, \quad t \geq s, 
$$
so that $\limsup_{t\to +\infty}|v(t)| \leq  \lim_{s \to +\infty}\sup_{u\geq s}|g_2(u)|=0$.  

Finally, suppose that $\limsup_{t\to -\infty}|v(t)|e^{-at} >0$. Then, after realising the change of variables 
$v(t)= \psi(t)e^{at}$, we find that $\limsup_{t\to -\infty}|\psi(t)| >0$ and 
$$
\psi'(t) = -(a+p(t-\tau))\psi(t)+ (1-p(t))e^{-a\tau}\psi(t-\tau) +g_3(t), 
$$
where 
$$
g_3(t)= e^{-at}\left(g_1(t) - (\phi_0(t-\tau)-p(t-\tau))v(t) + (p(t)- \phi_0(t))v(t-\tau)\right), \quad g_3(-\infty)=0. 
$$
It is easy to check that the Floquet multiplicators of the homogeneous  equation 
\begin{equation}\label{sz}
z'(t) = -(a+p(t-\tau))z(t)+ (1-p(t))e^{-a\tau}z(t-\tau)
\end{equation}
can be obtained from the Floquet multiplicators of (\ref{lp}) after multiplying them by $e^{-a\omega}$. 
Thus equation (\ref{sz}) is exponentially dichotomic (i.e. it does not have multiplicators on the unit circle). In particular,   it does not possess nontrivial bounded solutions. On the other hand, since $\psi(t), \psi'(t)$ are 
bounded functions and $g_3(-\infty)=0$, we can find a sequence  $t_j \to -\infty$ such that $\psi(t+t_j)$ converges, uniformly on compact subsets of $\R$, to a bounded nontrivial solution of  (\ref{sz}).  The obtained  contradiction shows that actually $\psi(-\infty)=0$. 
\hfill $\square$
\end{proof}
\begin{cor}\label{cort5}
Suppose that $g \in X^a$. Then equation $Jw=g$ has a solution $w\in X^a$ if and only if 
$\int_0^\omega g_\infty(s)(p_*'(s)+p_*(s))ds=0$. 
\end{cor}
\begin{proof} Note that there exists a solution $w_*\in X^a$ of $Jw_*=g$ if and only if the equation 
$(Ju)(t) =  g_4(t)$ with 
$$
g_4(t)= \int^{+\infty}_te^{t-s} (Lg)(s)ds \in X^a, 
$$
has a solution $u_* = w_*- g \in X^a$.   Now, it is easy to see that 
$$
u'(t)  +\phi_0(t-\tau)u(t) - (1-\phi_0(t)) u(t-\tau) = -g_4(t)+g_4'(t), 
$$
where 
$$
-g_4(t)+g_4'(t)= - (1 + \phi_0(t-\tau))g(t) +(1-\phi_0(t)) g(t-\tau) \in X^a. 
$$
Applying Lemma \ref{L18}, we obtain the following solvability criterion for $Jw=g$: 
$$
0= <-g_{4,\infty}+g'_{4,\infty}, p_*>=  \int_0^\omega \left[-(1 + p(s-\tau))g_\infty(s) + (1-p(s)) g_\infty (s-\tau)\right] p_*(s)ds= 
$$
$$
 \int_0^\omega \left[-p_*(s) - p(s-\tau)p_*(s) + (1-p(s+\tau))p_*(s+\tau)\right]  g_\infty (s) ds=  \int_0^\omega \left[-p_*(s) - p'_*(s) \right]  g_\infty (s) ds. 
$$
This completes the proof  of Corollary \ref{cort5}. \hfill \qed
\end{proof}
\begin{remark} \label{R4} Lemma \ref{L18} and Corollary \ref{cort5} are analogous to Theorems 3.1 and 3.5 in \cite{HW}. Due to the use of the Hale-Lin theory  \cite{HL}, our 
proof of these results is  shorter than in \cite{HW}. 
\end{remark}
As we have mentioned, 
semi-wavefront solutions of  (\ref{pfe})  will be
obtained as perturbations of  the oscillating connection $\phi_0(t)$ of
(\ref{wri}).  Since these semi-wavefronts may converge,  as  $t \to -\infty$, to the periodic solutions 
with periods $\tilde \omega$ slightly different from the period $\omega$ of $p(t)$, it is convenient to 
introduce a new small parameter $\gamma$ measuring the difference between $\tilde \omega$ and $\omega$.  We will incorporate $\gamma$ through the change of variables  $Z(t) = y((1+\gamma)t)$, 
where $\gamma \in [-\gamma_*, \gamma_*]$ for some small $\gamma_*>0$. 
After setting $\epsilon_\gamma = \epsilon/(1+\gamma)$ and $\tau_\gamma = \tau/(1+\gamma)$, we obtain from (\ref{pfe}) that 
\begin{equation}\label{pfede}
\epsilon_\gamma Z''(t) +Z'(t) -(1+\gamma)Z(t-\tau_\gamma)(1-Z(t)) =0.
\end{equation} 
Thus  the function $w(t)= Z(t)-\phi_0(t)$ satisfies the equation  
\begin{equation}\label{dev}
\epsilon_\gamma w''(t)+w'(t)-w(t)=-(Lw)(t) - G(\vare,\gamma,w)(t),
\end{equation}
where 
$$
G(\vare,\gamma,w)(t)= \epsilon_\gamma\phi_0''(t) + (1+\gamma)w(t-\tau_\gamma)w(t)- \gamma [w(t-\tau_\gamma)(1-\phi_0(t))- \phi_0(t-\tau_\gamma)w(t)] +$$
$$
 (1-\phi_0(t))[w(t-\tau_\gamma) - w(t-\tau)] + w(t) [\phi_0(t-\tau_\gamma) - \phi_0(t-\tau)]-$$
 $$\gamma\phi_0(t-\tau_\gamma)(1-\phi_0(t))-
(1-\phi_0(t))(\phi_0(t-\tau_\gamma)-\phi_0(t-\tau)).
$$
Clearly,  $Lw, G(\epsilon,\gamma,w) \in X^a$ when $w \in X^a$. 
Similarly to  Section \ref{S5}, 
a bounded function $w:\mathbb{R}\to\mathbb{R}$ is a solution of
(\ref{dev}) if and only if
\begin{equation}\label{deva}
(Jw)(t)=H(\vare,\gamma, w)(t),\q t\in\mathbb{R},
\end{equation}
where 
$$
H(0,\gamma, w)(t):=  \int^{+\infty}_te^{t-s}G(0,\gamma,w)(s)ds,$$ 
and, for $\vare >0$, 
$$\displaylines{
H(\vare,\gamma,w)(t)=\int^{+\infty}_t \left[
\frac{e^{\be(\vare_\gamma)(t-s)}}{\sqrt{1+4\vare_\gamma}} - e^{(t-s)}\right](Lw)(s) \, ds\ +\cr
{1\over {\sqrt{1+4\vare_\gamma}}}\left [\int_{-\infty}^t{{e^{\al(\vare_\gamma)(t-s)}}}\left((Lw)(s) + G(\vare,\gamma,w)(s)\right)\, ds+ \int^{+\infty}_t
e^{\be(\vare_\gamma)(t-s)} G(\vare,\gamma,w)(s)
ds\right].}$$
After some lengthy but standard computations (cf. the proof of Lemma \ref{L3.5} above and  Propositions 2.1, 2.2 in \cite{DH} or Lemma 4.2 with Corollary 4.3 in \cite{HW}), we obtain the following
\begin{lem} \label{P1} Suppose that $w \in X^a$.\  Then  there exist  positive $\vare_*, \gamma_*$  such that   $H:[0,\epsilon_*] \times [-\gamma_*, \gamma_*]\times X^a\to X^a$ and $J:X^a\to X^a$  are continuous functions. Furthermore, 
for each $(\epsilon,\gamma, w)\in [0,\epsilon_*] \times [-\gamma_*, \gamma_*]\times X^a$ there exists $D_wH(\epsilon,\gamma, w)\in B(X^a,X^a)$ which depends continuously on  $(\epsilon,\gamma, w)$ in the operator norm $\|\cdot\|$ of  the Banach space $B(X^a,X^a)$ of all bounded linear homomorphisms of $X^a$. Finally, 
$H(0,0, 0)=0,\ D_wH(0,0, 0)=0$ and the kernel Ker\,$J \subset X^a$ of $J$ is finite-dimensional and nontrivial: Ker\,$J \ni \{c\phi_0', \ c \in \R\}$. 
\end{lem}
\begin{proof} Let $\vare_*, \gamma_*$ be small enough to satisfy $\beta(\vare_\gamma) \in (a,1)$ for all 
$(\vare, \gamma) \in  [0,\epsilon_*] \times [-\gamma_*, \gamma_*]$.   Obviously,  $H(0,0,0)=0$.  In addition, it is easy to see that $G,L:[0,\epsilon_*] \times [-\gamma_*, \gamma_*]\times X^a\to X^a$ are continuous functions. For instance,  the term $G_1(\vare,\gamma,w)(t) = (1+\gamma)w(t-\tau_\gamma)w(t)$ in the expression defining $G(\vare,\gamma,w)(t)$ is the composition of the continuous  (e.g. see Remark \ref{R3}) functions
$$
(-1,1)\times X^a \stackrel{\Gamma_1}\to \R\times\R \times X^a\times X^a  \stackrel{\Gamma_2}\to \R\times X^a\times X^a  \stackrel{\Gamma_3}\to X^a, 
$$
where $\Gamma_1(\gamma, w(\cdot)) = (1+\gamma, \tau_\gamma, w(\cdot), w(\cdot))$, $\Gamma_2(a,b,v(\cdot), w(\cdot)) = (a,v(\cdot-b), w(\cdot))$, 
$\Gamma_3(a, v(\cdot), w(\cdot)) = av(\cdot)w(\cdot)$. The continuity of $J$ follows from the estimate
$$
\left|\int_t^{+\infty}e^{t-s}f(s)ds\right|_a \leq \frac{3-a}{1-a}|f|_a, \ f \in X^a.  
$$
Similarly, for some positive $C$ which does not depend on $\epsilon$, we have that 
\begin{equation}\label{ve}
\left|\int_t^{+\infty}\left[\frac{e^{\beta(\epsilon)(t-s)}}{\sqrt{1+4\epsilon}}-e^{t-s}\right]f(s)ds\right|_a \leq C\epsilon |f|_a, \quad  \left|\int_t^{+\infty}e^{\alpha(\epsilon)(t-s)}f(s)ds\right|_a \leq C\epsilon |f|_a, \ f \in X^a.
\end{equation}
This guarantees the continuity of $H$ when $\epsilon \to 0$. 

Next, for $\vare >0$, we have  that 
$$\displaylines{
(D_wH_1(\vare,\gamma,w)h)(t)=\int^{+\infty}_t \left[
\frac{e^{\be(\vare)(t-s)}}{\sqrt{1+4\vare}} - e^{(t-s)}\right](Lh)(s) \, ds\ + {1\over {\sqrt{1+4\vare}}}\int_{-\infty}^t{{e^{\al(\vare)(t-s)}}}(Lh)(s)ds +\cr
{1\over {\sqrt{1+4\vare}}}\int_{-\infty}^t{{e^{\al(\vare)(t-s)}}} (D_wG(\vare,\gamma,w)h)(s)\, ds+ {1\over {\sqrt{1+4\vare}}} \int^{+\infty}_t
e^{\be(\vare)(t-s)} (D_wG(\vare,\gamma,w)h)(s)
ds=:}$$
$(\frak{L}_1(\vare)+\frak{L}_2(\vare)+\frak{G}_1(\vare,\gamma,w)+\frak{G}_2(\vare,\gamma,w))h$, where
$$
(D_wG(\vare,\gamma,w)h)(t)=  (1+\gamma)[w(t-\tau_\gamma)h(t) +w(t)h(t-\tau_\gamma)]  - \gamma [h(t-\tau_\gamma)(1-\phi_0(t))- \phi_0(t-\tau_\gamma)h(t)] +$$
$$
 (1-\phi_0(t))[h(t-\tau_\gamma) - h(t-\tau)] + h(t) [\phi_0(t-\tau_\gamma) - \phi_0(t-\tau)].$$
 If $\vare =\gamma=0$ then 
$$(D_wH(0,0, w)h)(t)=  \int^{+\infty}_te^{t-s}[w(s-\tau)h(s)+w(s)h(s-\tau)]ds, $$
so  $D_wH(0,0, 0)=0$. 

Now, the continuous dependence (in the operator norm) of $D_wH_1(\epsilon,\gamma, w)=\frak{L}_1(\vare)+\frak{L}_2(\vare)+\frak{G}_1(\vare,\gamma,w)+\frak{G}_2(\vare,\gamma,w)$ on  $\epsilon,\gamma, w$  is the most delicate part of the proof of Lemma \ref{P1}. Actually, it is easy to see that $\frak{L}_j(\vare), \ j =1,2,$ are continuous functions of $\vare$. However, in difference with \cite[Proposition 2.2]{DH},  $D_wG(\vare,\gamma,w)$ does not depend continuously on  $\epsilon,\gamma, w$ so that the continuity of $\frak{G}_j(\vare,\gamma,w)$ cannot be 
obtained as a consequence of the continuity of $D_wG(\vare,\gamma,w)$. 

Fortunately, the integration improves  the continuity properties of  $D_wG(\vare,\gamma,w)$.  Let us clarify  this statement by considering the following (most complicated and representative) term
$$(\frak{G}(\vare,\gamma,w)h)(t):= \int^{+\infty}_t
e^{\be(\vare)(t-s)} w(s)h(s-\gamma)ds$$
of the linear operator  $\frak{G}_2(\vare,\gamma,w)$ (other terms of $\frak{G}_j(\vare,\gamma,w)$ can be analysed in a similar way).   The first inequality of (\ref{ve}) indicates that $\frak{G}(\vare,\gamma,w)$ is continuous with respect to $\vare$ uniformly on 
$w(\cdot)h(\cdot -\gamma)$ from bounded subsets of $X^a$. This means that it suffices to prove that $\frak{G}(\vare,\gamma,w)$ depends continuously on $\gamma, w$. 
Set $q(s):= h(s-\gamma_1),\ \Delta \gamma = \gamma_1-\gamma_2$. Then it is not difficult to check the validity of the following estimates: 
$$
 \left|\int_t^{+\infty}e^{\be(t-s)}w(s)h(s)ds\right|_a \leq \frac{1}{\beta}|w|_\infty|h|_\infty+ \frac{2}{\beta -a}\left(
|w|_\infty|h|_a+ |w|_a|h|_\infty\right) \leq \frac{5|w|_a|h|_a}{\be -a}, 
$$
$$
|\frak{G}(\vare,\gamma_2,w)h-\frak{G}(\vare,\gamma_1,w)h|_a\leq \left|\int_{t+\Delta \gamma}^te^{\be(\vare)(t-s+\Delta \gamma)}w(s-\Delta \gamma)q(s)ds\right|_a+
$$
$$
|e^{\beta(\vare) \Delta \gamma} -1| \left|\int_{t}^\infty e^{\be(\vare)(t-s)}w(s-\Delta \gamma)q(s)ds\right|_a+ 
 \left|\int_{t}^\infty e^{\be(\vare)(t-s)}|w(s-\Delta \gamma)-w(s)|q(s)ds\right|_a\leq 
$$
$$
3e^{a(|\gamma_1|+|\gamma_2|)}|h|_a|w|_a\left\{e^{\beta|\Delta \gamma|}|\Delta \gamma|+ 4 e^{a|\Delta \gamma|}\left|\frac{1-e^{(\beta-a)|\Delta \gamma|}}{\beta -a}\right|\right\}+
$$
$$
\frac{15e^{a(|\gamma_1|+|\gamma_2|)}}{\be -a}|h|_a\left(|w(\cdot-\Delta \gamma)-w(\cdot)|_a+|e^{\beta(\vare) \Delta \gamma} -1| |w(\cdot-\Delta \gamma)|_a\right).
$$
Hence,  
$$
\|\frak{G}(\vare,\gamma,w_2)-\frak{G}(\vare,\gamma,w_1)\|\leq \frac{15e^{a|\gamma|}}{\beta-a}|w_2-w_1|_a,
$$
$$
\|\frak{G}(\vare,\gamma_2,w)-\frak{G}(\vare,\gamma_1,w)\|\leq C_1|\Delta \gamma| + C_2 |w(\cdot-\Delta \gamma)-w(\cdot)|_a,
$$
where $C_j = C_j(a, \gamma_1,\gamma_2, \beta, |w|_a), \ j=1,2,$ are locally bounded functions. Thus we can conclude that  
$\frak{G}(\vare,\gamma,w)$ is continuous with respect to $\gamma, w$ in the operator norm $\|\cdot\|$.

Finally, $Jw=0,\ w \in X^a$, if and only if  
\begin{equation}\label{wok}
w'(t)= -\phi_0(t-\tau)w(t) + (1-\phi_0(t))w(t-\tau).
\end{equation}
Thus $\phi_0' \in$\,Ker\,$J$. Recall now that equation (\ref{wok}) has a shifted dichotomy on $\R_-$ (with exponents 
$\alpha_1=0 < a < \beta_1$ and with one-dimensional strongly unstable space and with one-dimensional center manifold) and it is also exponentially stable on $\R_+$. Then Lemmas 4.5 and 4.6 from \cite{HL} assure that equation  (\ref{wok}) has at most two-dimensional space of 
solutions in $X^a$. \qed
\end{proof}
Corollary \ref{cort5} and Lemma \ref{P1} show that $J:X^a\to X^a$ is a Fredholm operator, so that the Lyapunov-Schmidt reduction can be applied to (\ref{deva}).
First, consider the subspace $Y^a\subset X^a$ defined by 
$$
Y^a=\left\{w \in X^a: \int_0^\omega w_\infty(s)p_*(s)ds=0\right\}.  
$$
Since $(\phi'_0)_\infty = p'(t)$ and $\int_0^\omega p'(s)p_*(s)ds =1$, we obtain $\phi'_0 \not \in Y^a$ and therefore  \cite{HW}  there exists a subspace 
$Z^a \subset Y^a$ such that 
\begin{equation}\label{deco}
X^a = \mbox{Ker\,}J \oplus Z^a. 
\end{equation}
It is clear  that $J: Z^a \to \mathcal{R}(J):= J(X^a)$ is a bijection so that
$J^{-1}:  \mathcal{R}(J) \to Z^a$ is a bounded linear operator due to the Banach open mapping theorem, cf. \cite[Lemma 4.4]{HW}.
Now, in order to find a complementary subspace of $\mathcal{R}(J)$ in $X^a$, consider the smooth function $\zeta(t)$ such that $\zeta(t)=p_*(t)$ for all $t \leq 0$ and $\zeta(t)=0$ for all $t \geq \omega$.  We have 
$\int_0^\omega\zeta_\infty(s) (p_*'(s)+p_*(s))ds =\int_0^\omega p^2_*(s)ds >0$ and therefore $\zeta \not \in \mathcal{R}(J)$ in view of Corollary \ref{cort5}. On the other hand, each $w \in X^a$ can be decomposed as follows 
$$
w = k\zeta + (w-k\zeta), \ \mbox{where}\ k = \int_0^\omega(p_*(s)+p'_*(s))w_\infty(s)ds\big/\int_0^\omega p^2_*(s)ds,
$$
and $P_\zeta w := k\zeta \in \{c\zeta, c \in \R\},$ $w-k\zeta = (I-P_\zeta)w \in\mathcal{R}(J)$. 

As a consequence, the question of the solvability of equation (\ref{deva}) in the space $X^a$ can be simplified to the question of the existence of solutions  $z\in Z^a$ of the system 
$$
z = J^{-1}(I-P_\zeta)H(\vare,\gamma, z), \quad P_\zeta H(\vare,\gamma, z) =0. 
$$
Due to Lemma \ref{P1}, $ D_wJ^{-1}(I-P)H(0,0,0) =0$ and therefore, by the implicit function theorem, 
there exists a continuous function $z= z(\epsilon, \gamma),$ $z: [0,\epsilon_1]\times [0,\gamma_1] \to Z^a$ such that $z(0,0) =0$ and  
\begin{equation}\label{lqq}
Jz(\epsilon, \gamma) = (I-P_\zeta)H(\vare,\gamma, z(\epsilon, \gamma)), 
\end{equation}
cf.  \cite[Lemma 4.6]{HW}.  Hence, in order to complete the proof of the existence of a periodic-to-point connections, it suffices to prove the existence of a continuous  function $\gamma: [0,\epsilon_2] \to \R, \gamma(0) =0, $ $\vare_2 \in (0,\vare_1),$ such that 
$$
P_\zeta H(\vare,\gamma(\epsilon), z(\vare, \gamma(\vare))) =0 \  \ \mbox{for all } \ \vare \in [0, \vare_2]. 
$$
So, let  $H_\infty, J_\infty$ denote operators obtained from $H, J$ as a consequence  of  the replacement of  the operators $G, L$ in the definition of $H, J$  with their limiting parts $G_\infty, L_\infty$: 
 $$L_\infty(w)(t)=(1+p(t-\tau))w(t) - (1-p(t)) w(t-\tau),$$ 
$$
G_\infty(\vare,\gamma,w)(t)= \epsilon_\gamma p''(t) + (1+\gamma)w(t-\tau_\gamma)w(t)- \gamma [w(t-\tau_\gamma)(1-p(t))- p(t-\tau_\gamma)w(t)] +$$
$$
 (1-p(t))[w(t-\tau_\gamma) - w(t-\tau)] + w(t) [p(t-\tau_\gamma) - p(t-\tau)]-
$$
 $$\gamma p(t-\tau_\gamma)(1-p(t))-
(1-p(t))(p(t-\tau_\gamma)-p(t-\tau)).
$$
Then, using the definition of $P_\zeta$, we can rewrite the equation $P_\zeta H(\vare,\gamma, z(\vare, \gamma)) =0 $ in the form 
\begin{equation}\label{ph}
P_\zeta H_\infty(\vare,\gamma, z_\infty(\vare, \gamma)) =0.
\end{equation}
Clearly,  (\ref{ph}) amounts to the equation 
$$
\Lambda(\vare,\gamma):=\int_0^\omega(p_*(s)+p'_*(s))H_\infty(\vare,\gamma,z_\infty(\vare, \gamma))(s)ds=0.
$$
Consequently, due to the implicit function theorem, it suffices to prove that $D_\gamma \Lambda(\vare,\gamma)$ exists and  is a continuous function defined in some neighborhood of $(0,0) \in \R_+\times \R$, and that 
$
D_\gamma \Lambda(0,0)\not=0. 
$
It should be noted here that the continuous differentiability of $\Lambda(\vare,\gamma)$ in $\gamma$ is not  a simple issue. Indeed, observe that 
the function  $H:[0,\epsilon_*] \times [-\gamma_*, \gamma_*]\times X^a\to X^a$ is not differentiable in $\gamma$.  A solution of this problem (proposed in \cite{Hale,HW}) is briefly outlined below, it uses a version of the parametric implicit function theorem, see \cite[Lemma 4.1]{hale}. 

First, from (\ref{lqq}) we obtain also that $z=z_\infty(\epsilon, \gamma)$ satisfies the equation
\begin{equation}\label{jz}
J_\infty z = (I-P_\zeta)H_\infty(\vare,\gamma, z), \ z(0,0) =0.
\end{equation}
It is also clear that 
$$\int_0^\omega z_\infty(\epsilon, \gamma)(s)p_*(s)ds=0,$$
so that $z_\infty(\epsilon, \gamma)$ belongs to the subspace $Y_\omega= \{w \in X_\omega: \int_0^\omega w(s)p_*(s)ds=0\}$ of the space $X_\omega$ of all continuous $\omega-$periodic functions with sup-norm. 
Obviously,  Ker\,$J_\infty= \{cp'(t), c \in \R\}$ that implies 
$X_\omega = \mbox{Ker\,}J_\infty \oplus Y_\omega. 
$

Next, applying to (\ref{jz})  the same arguments as in the case of equation  (\ref{ph}), we
conclude  that, for sufficiently small positive $\vare_3 < \vare_2, \gamma_2 < \gamma_1,$ there exists a unique continuous   solution $\hat z: [0,\epsilon_3]\times [0,\gamma_2] \to Y_\omega$ of equation (\ref{jz}).  Fortunately,  the above mentioned generalised implicit function theorem guarantees now that  $z(\vare,\gamma)$ is 
also continuously differentiable with respect to $\gamma$.
The uniqueness of solution in the space $Y_\omega$  implies that $z_\infty(\epsilon, \gamma)= \hat z(\epsilon, \gamma)$ and therefore $z_\infty(\epsilon, \gamma)$ is continuously differentiable  in $\gamma$.   See  \cite[Lemma 4.7]{HW} or \cite[Proposition 4.6]{DH} for more details.

Contrary to our expectancy, let us suppose now that $D_\gamma \Lambda(0,0)=0$.  Set $z_*(t)= D_\gamma z_\infty(0, 0)(t)$, after differentiating (\ref{jz}) consecutively with respect to $\gamma$ and $t$, 
we find that 
$$
z_*'(t)  = -p(t-\tau)z_*(t) + (1-p(t)) z_*(t-\tau) + p'(t) + \tau (1-p(t))p'(t-\tau).  
$$
This implies that the difference $d(t)= z_*(t)-tp'(t)$ satisfies the homogeneous equation 
$$
d'(t)  = -p(t-\tau)d(t) + (1-p(t)) d(t-\tau).  
$$
Now, since  $d(s+\omega)= z_*(s)-sp'(s) - \omega p'(s) = d(s) - \omega p'(s), \ s \in [-\tau,0]$, 
we conclude that $G_\R(1)$ contains two linearly independent functions $d,p'$ (this idea was  exploited in  the proof of Lemma 4.5 in \cite{DH} and Theorem 4.1 in \cite{hale}).  Thus $\dim G_\R(1) \geq 2$, which contradicts the hyperbolicity of the periodic solution $p(t)$.  

Hence, $
D_\gamma \Lambda(0,0)\not=0
$ and therefore there exists a continuous function $\gamma = \gamma(\vare), \ \gamma(0)=0$, such that  
$
 y(t,\vare) = \phi_0(t/(1+\gamma(\vare)))+ z(\vare, \gamma(\vare))(t/(1+\gamma(\vare)))
$
is the requested connection for (\ref{pfe}). 

Note that $y_\infty (t,\vare) = p(t/(1+\gamma(\vare)))+z_\infty (\vare, \gamma(\vare))(t/(1+\gamma(\vare)))
$. Then relations  (\ref{sinel}) and $z_\infty (0,0)=0$ suggest the sinusoidal form  of $y_\infty (t,\vare)$  \cite[p. 446]{mps2}.  The rigorous proof of this fact is given by Mallet-Paret and Sell in \cite{mps2}.  Indeed,  the change of variables $1-y(t) = e^{z(t)}$ transforms (\ref{pfe}) into the following {\it unidirectional monotone positive feedback system} 
$$
x_0'(t)= x_1(t), \quad \vare x_1'(t)= -\vare x_1^2(t)-x_1(t) + (e^{x_0(t-\tau)}-1). 
$$
The announced  sinusoidal property (invariant with respect to the change of variable $1-y =e^z$) of  nonconstant periodic solutions to such systems is established in Theorem 7.1 of \cite{mps2}.  This observation completes the proof of Theorem \ref{Te2A}. \qed
\begin{remark}
In fact, we believe that  $y_\infty (t,\vare)$ is a slowly periodic solution of (\ref{pfe}) in the spirit of the definition given in the second remark on p. 480 
of \cite{mps2} (and adapted for the positive feedback systems).   It should be noted that the concept of slow oscillations depends on the order and nonlinearities of system under consideration. In particular,  the definition of slowly oscillating periodic solution  given in the paragraph preceding Lemma \ref{L17} does not apply to equation (\ref{pfe}). 
\end{remark}
\begin{remark} Let some normalised kernel $K$ be fixed in (\ref{17nl}). By Alvaro and Coville results \cite{AC}, all fast semi-wavefronts are converging at $+\infty$ (this fact does not exclude their multiplicity). This means that we can expect the appearance of proper semi-wavefronts only for the moderate values of $c$. It would be quite interesting to find some explicit (e.g., in terms of the kernel $K$) estimates for the speed intervals where all three types of waves mentioned in Corollary  \ref{coroc}  exist. 
\end{remark}
\section*{Acknowledgments} 
\noindent  K. Hasik, J. Kopfov\'a and P. N\'ab\v elkov\'a were supported by the ESF project 
CZ.1.07/2.3.00/20.0002. S. Trofimchuk was partially supported by FONDECYT (Chile), project 1150480.

\end{document}